\documentclass[journal]{IEEEtran}

\usepackage{amssymb}
\usepackage{amsmath}
\usepackage{amsthm}
\usepackage{mathtools}
\usepackage{graphicx}
\usepackage{wrapfig}
\usepackage{subcaption}
\usepackage{enumitem}
\usepackage{caption}
\usepackage{float}
\usepackage{multicol}
\usepackage{cite}
\usepackage{hyperref}
\usepackage{xcolor}
\usepackage{eurosym}
\usepackage{dsfont}
\usepackage{bm}
\usepackage[ruled, vlined, linesnumbered]{algorithm2e}
\usepackage{multirow}

\newcommand{\tcb}{\textcolor{black}}
\newcommand{\tcbl}{}%\textcolor{blue}
\newcommand{\tcg}{}%\textcolor{blue} \textcolor{green}

\newtheorem{prop}{Proposition}

\newtheorem{lemme}{Lemma}
\newtheorem{defi}{Definition}

\DeclareMathOperator*{\argmax}{arg\,max}
\DeclareMathOperator*{\argmin}{arg\,min}
\newcommand{\up}{\Pi_{\text{up}}}
\newcommand{\mi}{\Pi_{\text{mid}}}
\newcommand{\MI}{\overline{\Pi_{\text{mid}}}}

\newcommand{\x}{P}
\newcommand{\y}{\alpha}
\newcommand{\X}{\mathcal{P}}
\newcommand{\Y}{\mathcal{A}}

\newcommand{\epsmi}{\varepsilon_{\text{mid}}}
\newcommand{\hub}{\mathcal{H}_{\text{cso}}}
\newcommand{\relaxed}{(21)}
\newcommand{\alphak}{(22)}

\begin{document}
\title{Hierarchical coupled driving-and-charging model of electric vehicles, stations and grid operators}

\author{
    
Benoit Sohet, Yezekael Hayel, Olivier Beaude and Alban Jeandin

\thanks{B. Sohet and Y. Hayel are with LIA/CERI, Univ. of Avignon, 84911, Avignon, FRANCE (e-mail: yezekael.hayel@univ-avignon.fr).}

\thanks{B. Sohet, O. Beaude are with EDF R\&D, OSIRIS Dept, EDF Lab' Paris-Saclay, 91120, Palaiseau, FRANCE (e-mail: \{benoit.sohet, olivier.beaude\}@edf.fr).}

\thanks{A. Jeandin is with Izivia, EDF group, 92419, Courbevoie, FRANCE (e-mail: alban.jeandin@izivia.com).\vspace{0.1 cm}}

}
%%%%%%%%%%%%%%%%%%%%%%%%%%%%%%%%%%%%%%%%%%%%%%%%%%%%%%%%%%%%%%%%%%%%%%%
\maketitle

\begin{abstract}
%Electric Vehicles' (EVs) growing number has various consequences, from reducing greenhouse gas emissions and local pollution to altering traffic congestion and electricity consumption. More specifically, 
The decisions of operators from both the transportation and the electrical systems are coupled due to Electric Vehicles' (EVs) actions.
Thus, decision-making requires a model of several interdependent operators and of EVs' both driving and charging behaviors.
Such a model is suggested for the electrical system in the context of commuting, which has a typical trilevel structure.
At the lower level of the model, a congestion game between different types of vehicles gives which driving paths and charging stations (or hubs) commuters choose, depending on travel duration and energy consumption costs. % at Wardrop Equilibrium (WE)
At the middle level, a Charging Service Operator sets the charging prices at the hubs to maximize the difference between EV charging revenues and electricity supplying costs.
These costs directly depend on the supplying contract chosen by the Electrical Network Operator at the upper level of the model, whose goal is to reduce grid costs.
This trilevel optimization problem is solved using an optimistic iterative algorithm and simulated annealing.
The sensitivity of this trilevel model to exogenous parameters such as the EV penetration and an incentive from a transportation operator is illustrated on realistic urban networks.
This model is compared to a standard bilevel model in the literature (only one operator).
\end{abstract}

\begin{IEEEkeywords}
Electric vehicles, Trilevel optimization, Smart charging, Coupled transportation-electrical systems
\end{IEEEkeywords}

%%%%%%%%%%%%%%%%%%%%%%%%%%%%%%%%%%%%%%%%%%%%%%%%%%%%%%%%%%%%%%%%%%%%%%%%%%%%%%%
%%%%%%%%%%%%%%%%%%%%%%%%%%%%%%%%%%%%%%%%%%%%%%%%%%%%%%%%%%%%%%%%%%%%%%%%%%%%%%%
%%%%%%%%%%%%%%%%%%%%%%%%%%%%%%%%%%%%%%%%%%%%%%%%%%%%%%%%%%%%%%%%%%%%%%%%%%%%%%%% 
\section{Introduction}
\tcg{
Electric Vehicles (EVs) are a promising solution to reduce greenhouse gas emissions and local pollution (air quality, noise).
Considering policies and targets around the world, EVs should account for 7~\% of the global vehicle fleet by 2030~\cite{EVoutlook}.
This represents an opportunity for the different stakeholders of electric mobility, but also challenges for the grid: in France for example, standard predictions give a 2.2 to 3.6 GW power demand increase during winter peak periods in 2035~\cite{RTE}.
Challenges already arise nowadays due to significant local penetrations of EVs\footnote{More than 80,000 EVs in circulation in Paris region:~\url{https://www.statistiques.developpement-durable.gouv.fr/sites/default/files/2020-04/immatriculations_neuves_2019.zip}}, which may lead to local grid constraints and infrastructure investment costs.
%like transformers aging and power losses.
Therefore, decision-making models are needed to help electric mobility operators with their infrastructure investments and pricing mechanisms, which exploit EV flexibility, in particular during charging.
}

\tcg{
Note that key components of the charging operation of EVs depend on their driving strategies, like the charging place and hours.
The driving and charging decisions of EV users (here referred to as ``EVs") are thus interdependent, which couples the electrical and the transportation systems, especially in urban networks.
This coupling is easily conceivable when during widespread holidays departures most of driving EVs need to charge at public charging stations, where there could
be significant waiting time and reduction of available power.
Therefore, due to EVs, infrastructure and pricing strategies of an operator of the transportation or the electrical system not only have an impact on the other operators of the same system, but also on the operators of the other system.
For example, Park \& Ride hubs installed at a city's outskirt by local authorities to mitigate traffic congestion and pollution\footnote{20,000 parking spaces at Paris gates:~\url{https://data.iledefrance-mobilites.fr/explore/dataset/parcs-relais-idf/}} are also an opportunity for ``smart charging".
}

\tcb{
Models of this coupled electricity-transportation system are suggested in works identified in the review paper\cite{wei2019interdependence} and in more recent papers.
In~\cite{zhang2020power} and~\cite{ammous2019joint}, some operator controls an EV fleet and solves the vehicle routing problem to minimize both EVs costs (travel and charging duration and cost) and grid costs.
Some papers focus instead on independent EV drivers who learn the optimal driving path and charging station to stop, like in~\cite{qian2019deep} or in one of our previous works~\cite{sohet2020learning}.
Review paper\cite{wei2019interdependence} distinguishes between expansion planning of charging stations as in paper~\cite{yao2014multi}, and coordinated operation of a fixed coupled electricity-transportation system, such as the present paper.
Among coordinated operation papers whose goal is to design price incentives, some are based on a real-time model of EVs like~\cite{tan2017}.
This often entails simplifications such as electrical grid constraints neglected in~\cite{li2020price}, or a coarse zone model of the transportation network in~\cite{yang2021dynamic}.
}

\tcb{
The present paper adopts a stationary EV model point of view in order to better focus on operators' long-term incentives, like the ones presented in the coordinated operation papers~\cite{alizadeh2016optimal,wei2017network} mentioned in review~\cite{wei2019interdependence}, and in more recent papers~\cite{shi2020distributed,qian2020enhanced}.
At the lower level, EVs behavior is modeled as the equilibrium of a driving-and-charging game:
EVs choose the resources (driving path, charging station\dots) with minimal costs -- either financial (traffic tolls, charging cost) or temporal (travel duration, queuing and charging times) -- which are function of the other EVs' strategies, due to congestion effects.
At the upper level, an  urban planner from the transportation and/or the electrical system incites these EVs through pricing mechanisms to adopt ``optimal" behavior. %short-term or long-term
}

\tcb{
However, the reduction in the literature of the electrical system's management to one type of operator is particularly unrealistic.
Concerning electric mobility, the electrical operators carry out two main functions: the Charging Service to EVs (guaranteed by Operators called CSOs) and the management of the Electrical Network (done by the ENO).
In this work, a CSO brings together both the charge point operator in charge of the station and the mobility service provider which deals with the EV customers, and the ENO is both the grid manager and the electricity provider.
In the previously mentioned papers, smart charging pricing is chosen to optimize either the ENO's~\cite{shi2020distributed} or the CSOs payoff~\cite{tan2017}, but the interaction between CSO and ENO is not considered.
In this work, we use instead a trilevel setting, with the EVs at the lower level, the CSOs at the middle one and the ENO at the upper level.
\tcbl{
In a future work, we will consider the interaction between several CSOs on top of EVs' game, as in papers~\cite{tan2017} and~\cite{yuan2015competitive}.
Other works such as~\cite{wu2015game, vagropoulos2015real} also consider several CSOs, but in a futuristic electricity market environment rather than the current realistic framework of CSOs buying electricity from suppliers (the ENO in the present paper).
}
}

In electrical systems, trilevel frameworks are commonly employed in cyber security~\cite{alguacil2014trilevel}, expansion planning~\cite{jin2013tri} or demand-side management~\cite{aussel2020trilevel}, but to our knowledge, only two papers on electric mobility use a trilevel setting.
In~\cite{shakerighadi2018hierarchical}, the ENO chooses the wholesale electricity prices for each charging station.
Each station charges its EVs, which only choose the charging quantity depending on the local retail electricity price set by the CSO of the corresponding station.
Due to the simple formulations of the three levels objective functions (no game between EVs), this trilevel setting is easily solved analytically.
In~\cite{alizadeh2018retail}, EVs choose a driving path, a station and a charging quantity.
The CSOs choose the local retail prices in order to minimize their costs (the electricity bought from the ENO) and the time EVs spend on the road.
The ENO chooses the local wholesale prices for each station to minimize its costs (related to electrical grid constraints) and the time EVs spend on the road.
Note that the lower level is not a game but simply an optimization problem as there is no interaction between EVs.
%The retail prices optimal for the CSOs remove the only congestion cost for EVs (waiting time at charging stations), so that the lower level game becomes an optimization problem.
The trilevel optimization is solved iteratively: the ENO updates the wholesale prices, then the CSO uses an analytical expression to compute the optimal retail prices.
The theoretical and algorithmic details are not specified in this work.

\tcb{
The contributions of this paper can be summarized as follows.
First note that, although they were two original contributions of our previous paper~\cite{sohet2020coupled}, this work still relies on two features which are unique in the coupled electrical-transportation literature:
\begin{itemize}
    \item considering commuting and EVs charging during a whole working day gives the possibility for smart charging mechanisms on top of pricing incentives;
    \item a charging price at a given hub which depends on the smart charging load at this hub.
    This price is a congestion cost function which can be nonseparable (i.e., not only depends on congestion nearby, but on all over the network) thus requiring new theoretical results to study the uniqueness of the equilibrium of EVs' game.
\end{itemize}
The original contributions of the present paper are:
\begin{enumerate}
\item a realistic model of commuting and charging at work using a trilevel setting, intended for and solved by the ENO, at the upper level.
The CSO and ENO maximize their payoffs using realistic pricing mechanisms and EVs interact both while driving and charging in a coupled game;
  \item a new theoretical proof of the unique aggregated charging need at each hub at the equilibrium of the coupled routing-and-charging game between EVs.
  \item a carefully designed iterative algorithm solving the trilevel model using simulated annealing, Brent's method and convex optimization, with a theoretical proof of the global algorithm's convergence;
  \item sensitivity results on a realistic setting and a comparative study of our trilevel model with a bilevel setting (ENO and CSO combined together in a unique operator using Locational Marginal Pricing), standard in the literature~\cite{alizadeh2016optimal, wei2017network}.
  %conducting numerical tests on a realistic setting.
\end{enumerate}
}
The paper is organized as follows.
The objectives and available strategies of the three types of agents considered (EVs, CSO and ENO) are introduced in Sec.~\ref{sec:trilevel_model}.
The theoretical trilevel model of the interactions between these agents is given in Sec.~\ref{sec:theory}.
An algorithmic solution of this trilevel optimization problem is studied in Sec.~\ref{sec:algorithm} and applied in Sec.~\ref{sec:numerical} to examine the sensitivity of our model to exogenous parameters and compare it to the standard model in literature.
Finally, conclusions and perspectives are given in last section.

\begin{table}[]
\centering
\caption{Table of main notations}
\begin{tabular}{cl}
\multicolumn{2}{l}{\textbf{Abbreviations}}\\
\hline
CSO & Charging Service Operator \\
ENO & Electrical Network Operator\\
EV / GV & Electric / Gasoline Vehicles\\
%  SO & System Operator (= CSO + ENO)\\
Hub & Park \& Ride charging station\\
LMP & Locational Marginal Price\\
P\&C & Plug and Charge\\
PT & Public Transport\\
SC & Smart Charging\\
% SoC & State of Charge\\
WE & Wardrop Equilibrium
\vspace{1mm}\\
\multicolumn{2}{l}{\textbf{Parameters}}\\
\hline
$i_r$& Parking (and charging) hub associated to path $r$\\
$r_S$& Path $r$ + charging at hub\\
$r_H$& Path $r$ + charging later (e.g., at home)\\
$\hub$& Set of CSO's hubs\\
$e_1$& EV class that can charge at hub or later\\
$e_0$& EV class that can only charge at hub\\
$X_e$& EV penetration\\
$t_i$& PT fare from hub $i$ to destination\\
$\lambda_i$& Charging unit price~\eqref{eq:LMP} at CSO's hub $i\in\hub$\\
$\lambda^0_S$& Constant charging unit price at city's hub $i\in \mathcal{H}$\textbackslash$\hub$\\
$\lambda^0_H$& Constant charging unit price at home\\
$L_i$& Charging need aggregated over all EVs charging at hub $i$\\%\bm{L} ?
$\mi$& CSO's objective (charging revenues $-$ supply contract)\\
$\up$& ENOS's objective (supply contract $-$ grid costs)
\vspace{1mm}\\
\multicolumn{2}{l}{\textbf{Variables}}\\
\hline
$x_{s,r}$& Flow rate of vehicle class $s$ on path $r$\\
$x_{s,a}$& Flow rate of vehicle class $s$ on arc $a$ ($=\sum_{\{r \text{ s.t. } a\in r\}} x_{s,r}$)\\
$\ell_{i,t}$& Aggregated charging power at hub $i$ and time slot $t$\\ %\ell_{i,t}^0 ?
$\alpha$& Charging unit price magnitude (CSO's decision variable)\\
$P$& Elec. supplying contract threshold (ENO's decision variable)
\end{tabular}
\label{tab:notations}
\end{table}

\section{A smart coupled driving-and-charging model with three types of actors}
\label{sec:trilevel_model}

The smart charging use case considered in this work is about commuting:
drivers, coming from different places, choose their path to get to their workplace, which are all located in a same city or urban area.
In this city, there are Park \& Ride hubs where EV users may leave their car charging during working hours, and finish the commuting to their workplace by foot or public transport.
In addition to drivers, there are two other types of agents/operators considered in this system:
\begin{itemize}
    \item The CSOs which are in charge of several hubs and decide the corresponding smart charging fares;
    \item The ENO which is in charge of the grid of the city considered (assumed to be a medium-voltage one) and which specifies the electricity supply contract with CSOs.
\end{itemize}
\tcb{
Note that the operators do not control vehicles (in the sense of Vehicle Routing Problems) but only send incentives to influence both the driving and charging decisions of drivers (who interact through congestion effects in the sense of routing games).
}

 %======================================================
% 
\subsection{Vehicle users: a coupled driving-and-charging decision}

The transportation network is modeled by a graph in which each arc represents a street (illustrated in Fig.~\ref{fig:schema}).
Here a path $r$ refers to the successive arcs used to go from an origin $O$ to the hub $i_r$ chosen to park the vehicle, and also includes the public transport arc connecting $i_r$ to the workplace destination $D$.
Vehicle users have to choose one of the path to go from their origin to their destination, depending on the commuting duration and on the energy consumption costs.

\tcg{
Vehicles are of two distinct types: EVs (index $e$) and Gasoline Vehicles (GVs, index $g$) which rely on thermal engines.
EVs are split into two classes: EVs in class $e_1$, when choosing a path $r$, can either decide to charge at hub $i_r$ during working hours (fictitious path denoted $r_S$), or only park there and charge later, e.g. at home (path $r_H$). %, see Fig.~\ref{fig:transport}
EVs in class $e_0$ do not have enough energy (their State of Charge, or SoC, is low) to go home after work and will automatically choose to charge at the hub (path $r_S$).}
Vehicles of a same class ($g$, $e_0$ or $e_1$) share the same costs, but more vehicle classes could be considered in order to distinguish for example pure EVs from plug-in hybrid vehicles.

The duration cost of a path $r$ is the same for all vehicle classes and is made of two parts.
The first one reflects congestion on each road $a$ composing path $r$ following the Bureau of Public Road (BPR) function~\cite{BPR}: %\tcb{with updated parameters~\cite{jeihani2006improving}:}
\begin{equation}
    d_a(x_a) = \tau \frac{l_a}{v_a} \left(1+2\left(\frac{x_a}{C_a}\right)^4\right)\,,
\end{equation}
with $x_a = x_{g,a}+x_{e_0, a}+x_{e_1, a}$ the total flow of vehicles of all classes on arc $a$, $l_a$ its length, $v_a$ the corresponding speed limit and $C_a$ its capacity. \tcb{The internal parameters of the BPR function are determined in accordance with \cite{jeihani2006improving} for urban area congestion measures.}
The value of time $\tau$ transforms the travel duration into a monetary cost.
Note that this congestion cost depends on the drivers path choice through variable $x_a$.
The second part of the duration cost is a constant $t_i$ representing, if any, the time (expressed as a cost) to go from the hub $i$ where a vehicle is parked to its workplace. %(see Fig.~\ref{fig:transport})
Note that other constant costs can be added to $t_i$ like public transport fares.

\begin{figure}
    \centering
    \includegraphics[width = 0.5\textwidth]{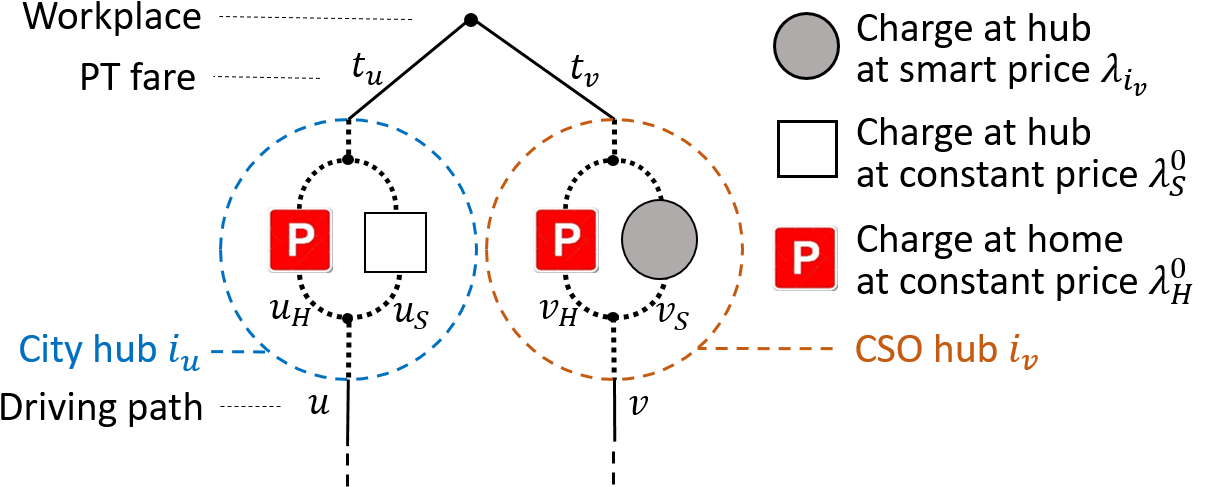}
    \caption{\tcb{Illustration of a transportation network.
    \textit{
    %There are two paths $a$ and $b$ leading to a workplace destination.
    Each path $r \in \{u,v\}$ includes the driving path to get to the hub $i_r$ associated to $r$ and the PT fare $t_r$ to go from hub $i_r$ to the workplace.
    At hub $i_r$, it is possible to only park there and charge later at constant price $\lambda^0_H$ (the corresponding global path is written $r_H$); or, it is possible to charge at the hub ($r_S$).
    Considering the latter decision, the charging price at hub $i_r$ is constant ($\lambda^0_S$) if the hub is managed by the City authority (like hub $i_u$).
    Otherwise, if the hub is managed by the CSO (like hub $i_v$) then the charging price $\lambda_{i_v}$ is smartly designed by the CSO.}}}
    \label{fig:schema}
\end{figure}

\tcg{
The second type of cost for drivers is related to energy consumption.
The charging fare at hub $i$ is more precisely a charging unit price $\lambda_i$, i.e. per unit of energy used, and is specified in the next section.
EVs deciding to charge during working hours will be charged up to full SoC. More precisely, the amount of energy EVs of class $e_j$ charge at the hub is equal to the energy consumed while driving to their workplace, plus the difference $s_j$ between full SoC and the SoC before the morning trip.} The former quantity of energy is assumed to depend only on the travelled distance, i.e. the energy $m_s$ consumed by a vehicle of class $s$ per distance unit is constant.
Thus, EVs of class $e_j$ charging at the hub $i_r$ of path $r$ have to pay:
\begin{equation}
    \ell_{e_j, r}\times \lambda_{i_r}\,,\quad\text{with } ~~ \ell_{e_j, r} = \left(l_r m_e + s_j\right) \,,
\end{equation}
where $l_r$ is the total length of path $r$.
\tcb{
Then, the energy consumed by an EV on path $r$ is approximated by the product $l_rm_e$. A more realistic consumption model -- which also depends on the driving speed and exogenous weather conditions (for auxiliary consumption) -- would be an interesting follow-up of this work.
}
\tcg{It is assumed that EVs which do not charge at the hubs also take into account a consumption cost: $\ell_{e_j,r} \lambda^0_H\,,$ with $\lambda^0_H$ a constant corresponding to the charging unit price at home for example.
Similarly, the consumption cost for GVs is $\ell_{g,r} \lambda_g$ with $\ell_{g,r}=l_r m_g$.}
The total cost for a vehicle of class $s$ choosing path $r$ is:
\begin{equation}
    c_{s, r}(\bm{x}) = \sum_{a\in r} d_a(x_a) + t_{i_r} + \ell_{s, r} \lambda\,,
    \label{eq:costs}
\end{equation}
where $\lambda$ is equal to $\lambda_g$ if $s=g$, $\lambda_H^0$ if $s=e_1$ and $r=r_H$ or $\lambda_{i_r}$ if $s=e_j$ and $r=r_S$.

The interaction between drivers through congestion effects constitutes a nonatomic multiclass congestion~\cite{JIANG14} game $\mathbb{G}$ with nonlinear cost functions $\bm{c} = \left(c_{s,r}\right)$ defined in~\eqref{eq:costs}.
In such frameworks, the vehicle users reach a particular distribution of choices between the possible paths, called a Wardrop Equilibrium (WE), where no user has an interest to change
her choice unilaterally:

\begin{defi}[\textbf{Wardrop Equilibrium~\cite{Wardrop52}}]\label{WE}
The global vehicle flow $\bm{x}^*$ is a Wardrop Equilibrium (WE) if and only if:
\begin{equation}
\forall s \in \{g, e_0,e_1\}\,, \quad c_{s,r}(\bm{x}^*)\leq c_{s, r'}(\bm{x}^*),
\end{equation}
for all paths $r, r'$ with $r$ such that $x^*_{s,r}>0$.
\label{defi:we}
\end{defi}

\noindent
The charging unit price $\lambda_i$ at CSO's hub $i\in\hub$ is a congestion cost determined by the CSO and is specified in the next section.

\subsection{Charging Service Operator: sets charging price}
\label{sec:CSO}
\tcg{
A CSO adapts the charging unit prices at its hubs in order to maximize the difference between its revenues from EV charging and its electricity supplying costs.
Here, it is supposed that there is only one CSO in the city to avoid a complex competition between several CSOs, which will be the focus of a future work.
More precisely, this CSO does not own all the hubs of the city, otherwise it could set arbitrarily high prices and EVs of class $e_0$ would have no choice but to pay these prices.
Instead, some hubs belong to the city for example with a constant charging unit price $\lambda^0_S$, supposed higher than $\lambda^0_H$, the one available at home.
The set of all hubs is denoted $\mathcal{H}$ and the set of the CSO's hubs, $\mathcal{H}_{\text{cso}}$.}
%, like in Paris\footnote{\url{www.iledefrance-mobilites.fr/actualites/18-000-places-de-parc-relais-2018/}},

\tcg{
At its hubs, the CSO determines the charging profile over working hours aggregated over all EVs such that their SoC is full at the end of the day.
The charging unit price $\lambda_i$ at each CSO's hub $i$ can be lower than the one at city's hubs, $\lambda^0_S$. % which are managed for example by a city authority.
More precisely, $\lambda_i$ depends on the total charging need of EVs at CSO's hub $i$ and other electricity usages called nonflexible because of their nonshiftable operation. %the charging unit prices at CSO's hubs
This nonflexible term corresponds for example to the consumption of a shopping mall attached to the hub.
%like local renewable electricity generation if hubs are equipped with solar panels for example.
The CSO schedules EV charging in order to smooth the power load at its hubs and therefore reduce its electricity supplying costs (see next section for details). %the costs of the ENO, therefore lowering its
For each CSO's hub, the aggregated charging need is scheduled using a water-filling algorithm introduced in~\cite{mohsenian10}. 
%Note that other algorithms introduced in an earlier work~\cite{sohet2020isgt} can be used instead, like a water-filling applied to the whole set $\mathcal{H}_{\text{cso}}$.
% To incite EVs to reduce its costs, the CSO then sets the charging unit prices $\lambda_i$ in function of the results of the optimal charging scheduling profiles.
% Using a function structure for $\lambda_i$ instead of the constants used in most papers (\cite{alizadeh2016optimal, wei2017network}) offers better convergence properties and robustness.
The CSO sets the charging unit prices $\lambda_i$ based on the output of this algorithm. % minimizing operational costs.
% This smart pricing scheme offers better convergence properties and robustness of the overall solution of the optimization problem, instead of flat rate pricing schemes used in most papers (\hspace{-0.01mm}\cite{alizadeh2016optimal, wei2017network}).
% A more thorough comparison of these different pricing schemes will be the focus of a future work.
}

The working hours are divided into $T$ discrete time slots.
Without loss of generality, these time slots have the same duration of a time unit.
For each hub $i$, the CSO has to determine the energy $\ell_{i,t}$ which is charged during time slot $t$.
% portion $\ell_{i,t}$ of the total charging need $L_i$
Assuming that the charging power is constant during each time slot,  $\ell_{i,t}$ also represents the constant power load of the charging operation at time slot $t$, as the duration of a time slot is a time unit.
The total charging need $L_i$ at hub $i$ is equal to:
\begin{equation}
    L_i(\bm{x_e}) = \sum_r \delta_{i_r,i}\sum_{j=0,1} x_{e_j, r_S} ~\ell_{e_j,r}\,,
\end{equation}
with $\delta_{i_r,i}=1$ or 0 whether or not the destination hub $i_r$ associated to path $r$ is hub $i$, and $x_{e_j, r_S}$ the flow of EVs of class $e_j$ choosing path $r$ and charging at hub $i$ at the end of the path.
Each hub $i$ has its own nonflexible consumption of electricity $\ell^0_{i,t}$ at time slot $t$.
Note that this nonflexible term can include local electricity production and be negative, but here it is supposed that $\ell^0_{i,t}\geq 0$ to simplify notations.
The water-filling algorithm minimizes a quadratic proxy~\cite{mohsenian10} of the total load at hub $i$ while making sure that all EVs leave the hub with full SoC:
\begin{equation}
    G^*_i= \min_{\left(\ell_{i,t}\right)} \sum_{t=1}^T \left(\ell_{i,t}+\ell_{i,t}^0\right)^2\
    \quad \text{s.t.} \sum_{t=1}^T \ell_{i,t} = L_i\,.
    \label{eq:local}
\end{equation}
\tcb{Note that the Vehicle to Grid (V2G) technology can be integrated into this model by allowing $\ell_{i,t}$ to be negative, while making sure that at each time slot the aggregated SoC remains inside tangible bounds\footnote{\tcb{Min/max capacity of the ``aggregated battery'' connected to a given hub, where the max. bound is the sum of capacities of individual EVs plugged-in.}}.
However we chose not to consider V2G because in general injected electricity is not compensated financially yet and may be potentially harmful for the local distribution grid.
Note also that battery health limitations (depth of discharge, number of cycles\dots) cannot be integrated as it is because only the aggregated charging loads at each hub are modeled. %EVs are not considered separately but as a continuous mass (in the sense of "nonatomic games'')
A way to consider it however should be to assign a different charging power limit to each EV class depending on its initial SoC.}
Assuming without loss of generality that $\left(\ell^0_{i,t}\right)_t$ is increasingly sorted, the water-filling solution of this problem depends on the aggregated charging need $L_i$:
\begin{equation}
    G^*_i(L_i) = \frac{\left(L_i + L^0_{i,t_0}\right)^2}{t_0(L_i)} + \sum_{t=t_0+1}^T \left(\ell^0_{i,t}\right)^2\,,
    \label{eq:val}
\end{equation}
where $L^0_{i,t} = \sum_{s\leq t}\ell^0_{i,s}$ and $t_0(L_i)\geq 1$ is such that $L_i\in ]\Delta_{t_0};\Delta_{t_0+1}]$, with $\Delta_t = t\times\ell^0_{i,t}-L^0_{i,t}$ \tcb{for $t\leq T$ and $\Delta_{T+1} = +\infty$}.
The corresponding optimal aggregated charging profile is $\ell_{i,t}^*=0$ for $t>t_0$, and for $t\leq t_0$:

\begin{figure*}[ht]
\centering
\includegraphics[width = .95\textwidth]{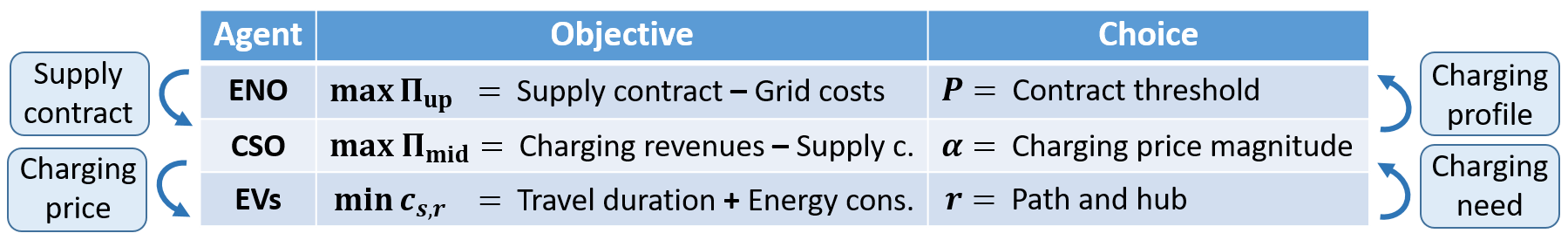}
\caption{Diagram of the different agents, their decision variables and their interactions.}
\label{fig:interactions}
\end{figure*}

\begin{equation}
    \ell_{i,t}^*(L_i) = 
    \frac{L_i+L^0_{i,t_0}}{t_0(L_i)} - \ell^0_{i,t}\,.
    \label{eq:water_filling}
\end{equation}

% \tcr{
% Then, the CSO chooses the charging unit price $\lambda_i$ at each hub $i$ to maximize its payoff.
% Unlike in most papers\cite{alizadeh2016optimal, wei2017network} where $\lambda_i$ is an optimized constant, in this work $\lambda_i$ has a fixed function structure.
% The advantages, which will be the focus of a future work, are better convergence properties and a reduced sensitivity to parameters (which are taken into account in the function).
% Here, $\lambda_i$ is a function of $G^*_i$, so that a lower charging unit price corresponds to a higher payoff for the CSO.
% }
Then, the CSO sets the charging unit price $\lambda_i$ at hub $i$ as a function of $G^*_i$.
In this work we use the Locational Marginal Pricing (or LMP), in which $\lambda_i$ is the derivative of $G^*_i$, which is proven to be the most efficient way to incite users to reduce $G^*_i$~\cite{li2013distribution}.
Our model can be adapted to other pricing mechanisms, like the average CSO's cost~\cite{sohet2020coupled}.
More precisely, $\lambda_i$ is set to be proportional to the LMP as follows:
\begin{equation}
    \lambda_i(\alpha\,,L_i) = \alpha \times\frac{\mathrm{d}G_i^*}{\mathrm{d}L_i} = 2 \alpha \frac{L_i + L^0_{i,t_0}}{t_0(L_i)}\,,
\label{eq:LMP}
\end{equation}
with $\alpha$ the variable with which the CSO optimizes its payoff.
This variable is the same for all CSO's hubs $i$ and can be seen as a conversion parameter from %: the CSO needs to convert the water-filling 
marginal energy costs $\mathrm{d} G_i^*/\mathrm{d}L_i$ (kW$^2$/kWh) of all of its hubs $i$ into reasonable monetary prices $\lambda_i$ (\euro/kWh) in order to maximize its payoff.
As $\lambda_i$ is a function of $L_i$, the charging unit price at CSO's hub $i$ is a congestion cost like travel duration, i.e. depends on the number of EVs charging at hub $i$.
Note that the CSO does not change the structure of the charging unit prices at its hubs (as locational marginal prices), but only their order of magnitude.

As $G_i^*$ is a nondecreasing function of $L_i$ (see~\eqref{eq:val}), variable $\alpha$ must be nonnegative in order to have $\lambda_i\geq 0$.
Moreover, it is assumed that some regulator sets an upper-bound $\overline{\alpha}$ to the CSO's decision variable.
%which ensures that no CSO's hub $i$ can charge more than $\lambda_i = \lambda_S^0$ if there is only one EV at this hub, with a charging need of $\max(s_0,s_1)$.
%L_{1\text{ev}}=
% \begin{equation}
%     \overline{\alpha}=
%     \frac{\lambda_S^0}{\max_{i\in\hub}\left(\frac{\text{d}G_i^*}{\text{d}L_i}(L_{1\text{ev}})\right)}\,.
% \end{equation}
% with $L$ the total energy capacity of a battery of one EV.
% For $\alpha\geq \overline{\alpha}$, we consider that almost no EV charge at CSO's hubs.
The feasible set of the CSO's strategy is denoted $\Y = \{\alpha\in\mathbb{R} ~|~ 0\leq \alpha \leq \overline{\alpha}\}$.
The CSO wants to optimize its net payoff, the difference between its revenues and its costs.
 Its revenues are what EVs pay to be charged at CSO's hubs, and its costs come from the electricity supplying contracts with the ENO, which are described in next section.

% where $f(P_{r,t}^{\max}) = P_{r,t}^{\max}\times\Delta t - (\ell^0_{r,t} - p^0_{r,t})$, with $\Delta t$ the (same) duration of each time slot $t$.
% Through this contract with the ENO, the CSO is incited to take into consideration the thresholds $P$ chosen by the ENO in function of the nonflexible electricity consumption and production at the moment ($\ell^0_{r,t} - p^0_{r,t}$).
% Note that the objective of the passive CSO is similar, except from the charging profiles $\bm{\ell}_a$ which are not the result of some water-filling algorithm: all EV there are charged during the first time slot ($\ell_{a,1} = L_a$)\tcr{, or with constant power throughout the day ?}.
%======================================================

\subsection{Electrical Network Operator: designs CSO supply contract}

\tcg{
In this urban framework, only the medium-voltage distribution grid and its operator the ENO are considered, and not the possible interactions with low-voltage distribution and the transmission grids.
This ENO specifies the electricity supplying contract with the CSO to engage grid costs reductions.
The CSO has to pay the ENO the supplying costs $C_{i,t}$ for the energy used to charge EVs at CSO's hub $i$ and time slot $t$.
The ENO determines one of the parameters of the contracts, a power threshold $P$, which is the same for all hubs and time slots.
Whether the total load at given hub and time slot is above or below this threshold $P$, the CSO's electricity bill varies.
% Note that in France, such power thresholds exist and are chosen by the medium-voltage customers\footnote{\url{https://www.enedis.fr/tarif-acheminement}} (corresponding to the CSO here).
}

%The energy demand at a hub is assumed to be supplied at a constant power for the whole duration $\delta$ of the corresponding time slot.

The total load at CSO's hub $i$ and time slot $t$ is made of the optimal aggregated charging profile given by the water-filling algorithm and the nonflexible part, and is equal to $\ell^{\text{tot}}_{i,t} = \ell_{i,t}^* + \ell^0_{i,t}$.
If the total load $\ell^{\text{tot}}_{i,t}$ is below the power threshold $P$, the price per energy unit is $\mu(P)$, otherwise, the unit price of the exceeding load is $\overline{\mu}(P)>\mu(P)$.
Functions $\mu$ and $\overline{\mu}$ are increasing: the higher the power threshold $P$ prescribed to the CSO, the higher the price per energy unit.
To simplify, linear functions are used: $\mu(P)=q P$ and $\overline{\mu}(P) = \overline{q} P$ with $\overline{q}>q$.
% Considering function $f_0(P) = P - \ell^0_{i,t}$,
%The final expression of $C_{i,t}$ is:
The total supplying costs for both the charging and the nonflexible consumption at hub $i$ and time slot $t$ are given by the following function $C$, and the supplying costs $C_{i,t}$ only due to charging are defined on a pro rata basis:
\begin{equation}
    C_{i,t}\left(L_i\,,P\right) = \frac{\ell^*_{i,t}(L_i)}{\ell^{\text{tot}}_{i,t}(L_i)}\times C\big(\ell^{\text{tot}}_{i,t}(L_i),P\big)\,, \quad \text{where}
\label{eq:Cit}
\end{equation}
\begin{equation*}
C\left(\ell^{\text{tot}}_{i,t},P\right) = \mu(P)\min\left(\ell^{\text{tot}}_{i,t}\,, P\right)
+ \overline{\mu}(P) \max\left(0\,, \ell^{\text{tot}}_{i,t}-P\right).
\end{equation*}
% \begin{equation}
%     \begin{aligned}
%  C_{i,t}\left(L_i\,, P\right) & =  q  P \min\big[
% \ell_{i,t}^*(L_i) \,, ~\max\left\{0\,,P- \ell^0_{i,t}\right\}\big]
% \\
% & + \overline{q}  P\max\big[0\,, \hspace{7mm}\ell_{i,t}^*(L_i)-\left(P- \ell^0_{i,t}\right) \big]
%  .%11mm
% \end{aligned}
% \label{eq:Cit}
% \end{equation}
Note that even if threshold $P$ is the same for all CSO's hubs, the supplying cost functions $C_{i,t}$ for EV charging are different due to the different nonflexible loads $\ell_{i,t}^0$ at each hub $i$. 

\tcg{
For each time slot $t$, the ENO's cost $\mathcal{G}_t$ is defined as the marginal grid costs associated with EV charging. 
The grid costs are modeled as a quadratic proxy of the apparent power at the head of the city's grid.
If $S_t$ is the apparent power required to meet the total energy demand $\left(\ell^{\text{tot}}_{i,t}\right)_i$ during time slot $t$, and $S_t^0$ the one corresponding to the nonflexible demand $\left(\ell^0_{i,t}\right)_i$ only, then the ENO's cost can be expressed as $\mathcal{G}_t= (S_t)^2 - (S_t^0)^2$.
Note that at hubs $j\notin\mathcal{H}_{\text{cso}}$ which do not belong to the CSO, EVs are supposed to plug and charge: $\ell_{j,1}^* = L_j$ and $\ell_{j,t}^*=0$ if $t > 1$.
The apparent power is obtained by solving the power flow equations from the Bus Injection Model~\cite{zhu2015optimization}, which correspond to the power balance at each bus (between the given power production/load $S_{0,k}$ at bus $k$ and power flows $S_k$ from/to the bus):
\begin{equation}
S_{0,k} = U_k \sum_{m\in X_k} \overline{Y_{k,m}}\overline{U_m} ~(= S_k)\,,
\label{eq:pf}
\end{equation}
with $U_k$ the complex voltage at bus $k$, $X_k$ the set of buses connected to bus $k$ and $Y_{k,m}$ the admittance of the line between buses $k$ and $m$.
}

The ENO's objective can then be expressed as:
\begin{equation}
\up\left(P\,,\bm{L}\right) = \sum_{t=1}^T\Big(\sum_{i\in\mathcal{H}_{\text{cso}}} C_{i,t}\left(L_i\,, P\right)
- \beta\times\mathcal{G}_t\left(\bm{\ell}_t^*\right)\Big)
,
\label{eq:obj_ENO}
\end{equation}
with $\bm{L} = \left(L_i\right)_i$ and $\beta$ a parameter which transforms $\mathcal{G}_t$ into a monetary cost.
The ENO's decision variable, the power threshold $P\geq 0$, is supposed to be bounded by $\overline{P}$ by some regulator.
%, the charging load obtained if all EVs would fully charge during the same time slot and at the same hub.
The feasible set of the ENO's strategy is denoted $\X = \{P\in\mathbb{R} ~|~ 0\leq P \leq \overline{P}\}$.
Note that the ENO's objective depends on $\bm{L}$, the result of drivers' strategies, which depends itself on ENO's decision variable $P$, as shown in the next section.
%Note that homes are not part of the considered electrical grid so that the charged quantity at home is not taken into account in the ENO cost.
The different agents, their decision variables and their interactions are summarized in Fig.~\ref{fig:interactions}.

%======================================================
\section{The trilevel optimization problem}
\label{sec:theory}
%======================================================
\tcb{
Last section introduced the three types of agents in our smart charging framework and their interactions.
This section focuses on the outcome of such a system.
% It is supposed that all agents want to minimize their own objective, but they might not have access to the same information nor choose their decision variable simultaneously.
% As a matter of fact, 
% \tcb{The following multilevel model is the one the ENO solves to maximize its objective $\up$.}
The following multilevel optimization problem is solved by the ENO as the decision maker at the upper level of the decision process. In particular, the ENO aims to maximize its objective function denoted $\up$.
Note that the electricity supplying contract between the ENO and CSO, and the charging unit prices at CSO's hubs are long-term strategies (resp. of the ENO and CSO).
They are assumed to be based on the forecast of drivers' behavior on a specific working day, forecast which is the Wardrop Equilibrium (WE) vehicles naturally reach and which depends on the charging unit prices (see next section).
For example, the ENO might be pessimistic and optimizes its net payoff on a worst-case-scenario day (e.g., with a high proportion of EVs on the roads).
}
%\tcb{The forecast of drivers' behavior is given by the Wardrop equilibrium which is introduced in next section and which depends on the charging unit prices.}

\tcb{
The information available for each agent is as follows.
The drivers know their costs functions on this specific working day: they observe the charging unit price functions chosen by the CSO.
Therefore they can choose the optimal path and place to charge during this working day, corresponding to the WE of this day.
\tcbl{The CSO has access to the behavior model of vehicle users and knows the main characteristics of the problem, such as the transportation network properties, the travel demands between origins $O$ and destinations $D$, etc.}
Therefore, the CSO can compute the WE for any charging unit prices it chooses.
However, the CSO has no information on the grid topology and consequently on ENO's costs, so that it does not know how the ENO chooses the supplying contract.
Thus, the CSO must observe its supplying contract only once it is chosen by the ENO.}
\tcbl{Finally, the ENO has also access to the behavior model of vehicle users and to general information (e.g., travel demands), including the structure of the charging unit prices, which is assumed to be publicly disclosed by the CSO.}
%all the vehicles' costs functions, including the structure of the charging unit prices, which is assumed to be publicly disclosed by the CSO.
This way, the ENO can compute the WE, the CSO's revenues and then CSO's reaction to its supplying costs (chosen by the ENO).
%is assumed to know how the CSO reacts to each value of the supplying contract.
%This requires that the ENO knows how EVs behave and react to CSO's decision variable and knows the expression of the charging unit prices (which are publicly disclosed by the CSO, to inform EVs).
%On the contrary, the CSO has no information on the grid topology and therefore on ENO's costs, so that it cannot foresee ENO's choice and can only solve the bilevel model made of the medium and lower levels.
This constitutes a trilevel optimization problem as illustrated on Fig.~\ref{fig:interactions}, with the ENO at the upper level, the CSO at the middle one and the drivers at the lower level.

%Agents observe and react to the strategies of the others.
% The ENO knows the reactions of the CSO and of the drivers to its decision variable $P$, and chooses it accordingly.
% In function of the supplying contract $P$ chosen by the ENO, the CSO sets its decision variable $\alpha$ specifying the charging unit prices at its hubs, knowing the reactions of drivers to each $\alpha$. %both the ENO and CSO choices
% Finally, drivers observe the charging unit prices and choose the optimal path and place to charge during the working day under consideration.
% This constitutes a trilevel optimization problem as illustrated on Fig.~\ref{fig:interactions}, with the ENO at the upper level, the CSO at the middle one and the drivers at the lower level.

%======================================================
\subsection{Vehicle users at Wardrop Equilibrium}
%======================================================
Before defining the trilevel optimization problem, some details about the lower level are needed.
On the working day considered, the city's commuters have to choose how to get to their workplace and whether they charge their vehicle during the working hours. %(see Fig.~\ref{fig:transport}).
Due to the congestion effects on the road and also on the charging unit prices at CSO's hubs, the decision of a driver depends on the others'. The solution concept used to study this interaction is the Wardrop Equilibrium (see Definition \ref{WE}). Such equilibria can be computed via Beckmann function~\cite{BECKMANN56}: %using the following proposition:
%The particular structure of cost functions $\bm{c}$ enables us to

\begin{prop}\label{prop:Beckman}
For any CSO's strategy $\alpha$, the local minima~of the following constrained optimization problem are WE of $\mathbb{G}$:
\begin{equation}
    \min_{\bm{x}\in X} \mathcal{B}(\bm{x},\alpha)\,, ~ \text{ with}
    \label{eq:beckmann}
\end{equation}

 \begin{equation*}
 \begin{aligned}
%  \displaystyle
 &\mathcal{B}(\bm{x}, \alpha) = 
 \sum_a\hspace{-1mm}\int_0^{x_a}\hspace{-2mm} d_a%(x)\text{d}x 
 + \hspace{-2mm}\sum_{(s,r)\in \mathcal{S}}\hspace{-1mm}x_{s,r} \left(t_{i_r} + \ell_{s,r}\lambda_{s,r}\right)
 +\alpha\hspace{-2mm}\sum_{i\in\mathcal{H}_{\text{cso}}}\hspace{-2mm} \mathcal{G}^*_i(\bm{x}_e)\\
  &\quad X \hspace{4.3mm} = \Big\{(x_{s,r})_{s,r} ~|~  x_{s,r}\geq 0\,, ~\sum_{r\in OD} x_{s,r} = X_s^{OD}\Big\}\,,
 \end{aligned}
 \end{equation*}
 with $x_a = \sum_{\{r \text{ s.t. } a\in r\}}\sum_s x_{s,r}$ the total vehicle flow on arc $a$, $\mathcal{S}=\{(e_j,r_S) ~s.t.~ i_r\notin\hub,~ (g,r),~ (e_1, r_H)\}$ and $X_s^{OD}$ the portion of class $s$ vehicles with origin $O$ and destination $D$.
 \label{prop:beckmann}
\end{prop}

Unfortunately, for some CSO's strategies $\alpha\in\Y$, there might be several minima of~\eqref{eq:beckmann} and therefore, several WE.
However, the following proposition shows (proof in Appendix~\ref{appendix:unique}) that even if there are several WE, they all lead to the same congestion $d_a^*(\alpha)$ on each road $a$ and the same total charging need $L_i^*(\alpha)$ at CSO's hub $i$.
Therefore, for given strategies $\alpha$ and $P$, the CSO and the ENO can expect a unique drivers' impact on their metrics, respectively on CSO's hubs and on the electrical grid.

\begin{prop}
\label{prop:WE_same_costs}
Let the CSO's strategy be any $\alpha\in\Y$.
Any different WE $\bm{x},\bm{y}$ of game $\mathbb{G}$ verify:
\begin{equation}
    \forall a\,,~ x_a = y_a \,,\qquad \forall i\in\hub\,,~ L_i(\bm{x})= L_i(\bm{y})\,.
    \label{eq:property}
\end{equation}
%At WE, the total charging need $L_i^*(\alpha)$ at each CSO's hub $i\in\hub$ is unique.
\end{prop}

Note that total charging needs at WE depend on $\alpha$, CSO's decision variable (see the expression of $\mathcal{B}(\bm{x},\alpha)$ in~\eqref{eq:beckmann}).
%Proposition~\ref{prop:WE_same_costs} is due to the nondecreasing property of congestion and consumption costs.
According to Prop.~\ref{prop:beckmann} and~\ref{prop:WE_same_costs}, any solution of optimization problem~\eqref{eq:beckmann} gives the unique $\bm{L}^*(\alpha)=\left(L_i^*(\alpha)\right)_{i\in\hub}$ at WE.

%======================================================

\subsection{The trilevel problem formulation}
%======================================================

As mentioned in Sec.~\ref{sec:CSO}, the objective $\mi$ of the CSO is the difference between its charging revenues and its electricity supplying costs.
At each CSO's hub $i$, the revenue $R_i$ is the product between the charging unit price $\lambda_i$ and the total charging need $L_i$ at this hub.
The CSO knows, for each $\alpha\geq 0$, that this need is the unique $L_i^*(\alpha)$ when drivers are at equilibrium, so that the revenue from hub $i$ can be written:
\begin{equation}
    R_i(\alpha\,, L_i^*(\alpha)) =   L_i^*(\alpha) \times \lambda_i\big(\alpha\,, L_i^*(\alpha)\big)\,.
\end{equation}
In function of ENO's strategy $P$, the CSO has to minimize over $\alpha\in\Y$ the following objective:
\begin{equation}
    \mi\left(\alpha,P,\bm{L}^*(\alpha)\right) =\hspace{-2mm} \sum_{i\in\mathcal{H}_{\text{cso}}}\hspace{-2mm}\bigg(\hspace{-1mm}R_i\left(\alpha\right)-\sum_{t=1}^T C_{i,t}\big(L_i^*(\alpha),P\big)\hspace{-1mm}\bigg)
    \label{eq:mid}
\end{equation}
% The optimal value $\alpha^*(P)$ is a function of the ENO's strategy $P$.
% Note that $\alpha^*$ may be arbitrarily high, as in some cases the best option for the CSO is to prevent any EV to charge at its hubs (when its contract with the ENO is too expensive for example), which results in $R_i = C_{i,t} = 0$ and thus $\mi(+\infty, P) = 0$.

For each CSO's strategy $\alpha\in\Y$, the ENO knows the global charging need $\bm{L}^*(\alpha)$ at WE.
However, as the objective function $\mi$ is not convex, $\mi$ might have several global optima $\alpha^*$.

\tcg{In this work, it is supposed that there is a minimal cooperation between the CSO and the ENO, which leads to an optimistic formulation of the multilevel problem.
This optimistic assumption states that for any ENO's strategy $P$, the global optimum $\alpha^*$ of~\eqref{eq:mid} which gives the highest ENO's objective $\up(P,\bm{L}^*(\alpha^*))$ is chosen.
}
Finally, the global trilevel optimization problem to solve is:

% \begin{subequations}
% \begin{align}
% &\max_{P}~ \up\Big(P\,, \bm{L}^*\big(\alpha^*(P\big)\Big)\,,\label{eq:min_3}\\
% &\quad\text{s.t.} \qquad  0\leq P\leq \overline{P} \,,\\
% &\qquad \qquad \alpha^*\left(P\right) = \argmax_{\alpha} ~\mi\left(\alpha\,, \bm{L}^*(\alpha)\,, P\right)\,,\label{eq:min_2}\\
% & \qquad\qquad \text{s.t.} \qquad  \alpha\geq 0\,,\\
% & \qquad\qquad \qquad \quad \bm{L}^*(\alpha) = \bm{L}\left( \argmin_{\bm{x}}~\mathcal{B}\left(\bm{x}\,, \alpha\right)\right)\,,\label{eq:min_1}\\
% &  \qquad\qquad  \qquad \qquad \text{s.t.}\qquad x_{s,r} \geq 0\,,~ \sum_r x_{s,r} = X_s\,.
% \end{align}
% \label{eq:trilevel}
% \end{subequations}

\begin{subequations}
\begin{align}
&\max_{P\in\X, \alpha^*\in\Y}~ \up\Big(P\,, \bm{L}^*(\alpha^*)\Big)\,,\label{eq:min_3}\\
&\qquad\text{s.t.} \quad  \mi\left(\alpha^*\,, P\,,\bm{L}^*(\alpha^*)\right) = \MI\left(P\right)\,,\label{eq:min_2}\\ %\,, \bm{L}^*(\alpha^*)
& \qquad\quad\quad \text{s.t.} \quad  \bm{L}^*(\alpha^*) = \bm{L}\left( \argmin_{\bm{x}\in X}~\mathcal{B}\left(\bm{x}\,, \alpha^*\right)\right)\,,\label{eq:min_1}
\end{align}
\label{eq:trilevel}
\end{subequations}
\noindent
where $\MI \left(\x\right) = \max_{\y\in\Y}\mi\left(\y\,,\x\,,\bm{L}^*(\alpha)\right)$ and function $\argmin$ returns the set of global minima of $\mathcal{B}$, which share the same $\bm{L}^*$ (see Prop.~\ref{prop:WE_same_costs}).

This trilevel problem can be seen as a Stackelberg game (between the upper and middle levels) with equilibrium constraints (lower level)~\cite{ehrenmann2004equilibrium}.
Note that depending on the information available to the ENO and CSO, other trilevel frameworks can be considered: if both the CSO and the ENO know the reactions of the other, they play in a simultaneous Nash game, with equilibrium constraints (lower level).
However, solving this Nash game with algorithms such as Best Response may not converge due to the equilibrium constraints.
%If only the CSO knows the ENO's reaction, the resulting problem is~\eqref{eq:trilevel} with the CSO at the upper level and the ENO at the middle level.

\subsection{\tcb{Iterative method based on literature review}}
\label{sec:alternative}

\tcb{
In this section the most commonly used model of EV charging incentives in coupled electrical-transportation systems~\cite{alizadeh2016optimal,wei2017network} is introduced briefly.
In this reference model, the EV lower level is the same as the one in the new trilevel model introduced in this paper.
However, the different operators of the electrical system (CSO and ENO) are gathered into a unique System Operator (SO), which chooses the charging unit prices at its hubs which directly minimize the grid costs $\mathcal{G}=\beta\sum_t\mathcal{G}_t$ (instead of maximizing CSO's payoff).
To this end, the SO uses the following LMP function in order to determine the charging price for each hub $i$:
\begin{equation}
    \lambda_i(\bm{L}) = \tilde{\alpha} \times\frac{\mathrm{d}\mathcal{G}\left(\bm{\ell}\right)}{\mathrm{d}L_i}\,,
    \label{eq:LMP_true}
\end{equation}
which is the derivative of grid costs $\mathcal{G}$ obtained with power flow computations~\eqref{eq:pf} instead of the local quadratic proxy~\eqref{eq:local}.
Parameter $\tilde{\alpha}$ converts marginal grid costs into reasonable charging prices like CSO's decision variable $\alpha$.
However, no method to fix $\tilde{\alpha}$ is available in the literature.
}

\tcb{
Note that papers using this method have no smart charging algorithm, so here the whole EV battery need is assumed to be charged during the first time slot: $\bm{\ell}_1 = \bm{L}$ (method referred to as LMP+P\&C in numerical Sec.~\ref{sec:num}, for Plug and Charge).
However, it is possible to consider an improved method (referred to as LMP+SC, for ``Smart Charging'') by solving the following charging scheduling problem:
\begin{equation}
    \mathcal{G}^*= \min_{\left(\ell_{i,t}\right)} \sum_{t=1}^T \beta\mathcal{G}_t\left(\bm{\ell}_t\right)\
    ~~ \text{s.t.} ~\forall i\,,~\sum_{t=1}^T \ell_{i,t} = L_i\,.
    \label{eq:grid_aware_wf}
\end{equation}
Both alternative methods (LMP+P\&C and LMP+SC) follow an iterative process: they alternatively compute the aggregated charging needs $\bm{L}^{(0)}$ at WE corresponding to charging unit prices $\bm{\lambda}^{(0)}$, then compute $\bm{\lambda}^{(1)}$ using~\eqref{eq:LMP_true} then update the charging needs $\bm{L}^{(1)}$ and so on.
Note that there is no proof of convergence of this iterative process in the literature.
% This improved method is referred to as 
%(referred to as $\mathcal{M}_{\alpha}$
}

%%%%%%%%%%%%%%%%%%%%%%%%%%%%%%%%%%%%%%%%%%%%%%%%%%%%%%%%%%%%%%%%%%%%%%%%%%%%%%%

\section{Resolution of trilevel optimization problem}
\label{sec:algorithm}
%%%%%%%%%%%%%%%%%%%%%%%%%%%%%%%%%%%%%%%%%%%%%%%%%%%%%%%%%%%%%%%%%%%%%%%%%%%%%%%
\subsection{\tcb{An iterative method for upper and middle levels optimization}}
%%%%%%%%%%%%%%%%%%%%%%%%%%%%%%%%%%%%%%%%%%%%%%%%%%%%%%%%%%%%%%%%%%%%%%%%%%%%%%%
\tcb{
In most multilevel optimization problems, the convex lower level is replaced by the corresponding Karush-Kuhn-Tucker (KKT) conditions~\cite{colson2007overview}, which would transform the trilevel problem~\eqref{eq:trilevel} into a bilevel (upper-middle) one with equilibrium constraints.
However, using KKT conditions introduces integer variables and therefore transforms the global optimization problem into a mixed-integer nonlinear optimization problem, which increases dramatically the computational complexity~\cite{Papadimitriou98}.
In our setting we found that it was much faster to rather keep the initial trilevel structure~\eqref{eq:trilevel} and simply solve the convex lower level using sequential least squares programming~\cite{boggs1995sequential}.
%~\cite{kleinert2020there}
Thus, for the resolution of the global problem, we focus on the upper (ENO) and middle (CSO) levels.
The lower level is referred to as an implicit numerical function $\bm{L}^*(\alpha)$ of CSO's price strategy $\alpha$ (see Prop.~\ref{prop:WE_same_costs}), which is the global charging need when vehicle users are at equilibrium.
The global trilevel optimization problem is rewritten as:
\begin{equation}
\begin{aligned}
&\max_{\x \in \X\,, \y^*\in\Y}~ \up \left(\x\,,\bm{L}^*(\alpha^*)\right)\,,\\
&\qquad\text{s.t.} \quad \mi\left(\y^*\,, \x\,,\bm{L}^*(\alpha^*)\right) \geq \MI \left(\x\right) - \epsmi\,,
\end{aligned}
\label{eq:bilevel}
\end{equation}
with $\epsmi>0$ a tolerance level introduced to guarantee the convergence of the algorithm suggested to solve~\eqref{eq:bilevel}. % (see proof of Prop.~\ref{prop:mistros}).
Note that a pessimistic version of Algorithm~\ref{algo:Mitsos} introduced below can be used instead of the optimistic formulation~\eqref{eq:bilevel}.}
\tcb{To ease notations, $\mi\left(\y^*, \x,\bm{L}^*(\alpha^*)\right)$ is written $\mi\left(\y^*, \x\right)$, but note that both the computation of $\mi$ and $\up$ requires $\bm{L}^*$, i.e. to solve the convex lower level optimization.}

% As the objective function $\mi$ is not convex, the middle level problem might have several global optima.
% Here, an optimistic formulation of the bilevel problem is considered, as there might exist cooperation between the CSO and the ENO: the CSO chooses its optima maximizing the ENO objective.

% The optimistic bilevel problem considered can be written as~\cite{bard1983algorithm}:

% Note however that in formulation~\eqref{eq:bilevel}, although function $\MI$ is still unknown.
The global trilevel problem~\eqref{eq:bilevel} is solved using Algorithm~\ref{algo:Mitsos}, which is a simplified version of the iterative bounding algorithm introduced in~\cite{mitsos2008global}, as there are no constraints at the upper and middle levels other than variable bounds.
% At each iteration $k$ of the algorithm:
% \begin{enumerate}[label=(\roman*)]
%     \item A relaxed version of~\eqref{eq:bilevel} is solved:
%     \begin{equation}
% \left(\x_k\,,\y_k\right)=\argmax_{\x \in \X\,, \y\in\Y}~ \up \left(\x\,, \y\right)\,,
% \label{eq:relaxed}
% \end{equation}
% \begin{equation*}
% \text{s.t.} ~~\forall l<k\,, \quad \mi\left(\y\,, \x\right) \geq \mi \left(\overline{\y}_l\,,\x\right) - \frac{\epsmi}{3}\,,
% \end{equation*}
% where the $\overline{\y}_l$ are defined below, in the second step of each iteration;
% \label{item:1}
%     \item A new $\overline{\y}_k$ is computed:
%     \begin{equation}
%         \overline{\y}_k = \argmax_{\y\in\Y} ~\mi\left(\y\,,\x_k\right)\,.
%     \label{eq:alpha_k}
%     \end{equation}
% \label{item:2}
% \end{enumerate}
% Note that here function $\argmax$ means any optimal solution.
%The global iterative procedure is summarized in Algorithm~\ref{algo:Mitsos}.
In Algorithm~\ref{algo:Mitsos}, the global optimization problems~\relaxed~and~\alphak~at each iteration are solved by algorithms detailed in next section, but any other suitable algorithms can be applied.
By definition of $\overline{\alpha}_k$, if the solution of~\relaxed~at an iteration of Algorithm~\ref{algo:Mitsos} verifies the stopping criteria, then it is a solution of the initial trilevel problem~\eqref{eq:bilevel}. %, which is equivalent to, by definition of $\overline{\alpha}_k$:
% \begin{equation}
%     \mi\left(\y_K\,, \x_K\right) \geq \mi \left(\overline{\y}_K\,,\x_K\right) - \epsmi\,.
%     \label{eq:criter}
% \end{equation}
%Note that it works with any suitable global optimization algorithms for~\relaxed~and~\alphak.
The convergence of Algorithm~\ref{algo:Mitsos} is guaranteed by the following proposition (proved in Appendix~\ref{appendix:mistros}).
\begin{algorithm}[t]
\KwIn{$\x_0\,,\y_0\,, k=0$}
\textbf{Notation:} $\mi(\alpha\,,P) \leftarrow \mi\big(\alpha\,,P\,, \bm{L}^*(\alpha)\big)$\\
$\overline{\y}_0 = \argmax_{\alpha}\mi(\alpha\,,P_0)$\\% \hspace{29mm} \alphak\\
    \While{$\mi(\y_k\,,P_k) < \mi(\overline{\y}_k\,,P_k)-\epsmi$}
    {
    $k \leftarrow k+1$\\
    (i) $\left(\x_k\,,\y_k\right)=\argmax_{P,\alpha}\up(P,\bm{L}^*(\alpha))$\hspace{5mm} \relaxed\\%\;
    ~~~s.t. $\forall l<k$, $\mi\left(\y, \x\right) \geq \mi \left(\overline{\y}_l,\x\right) - \epsmi/3$\\
    ~~~solved with simulated annealing (\textbf{Algorithm~\ref{algo:annealing}})\\
    (ii) $\overline{\y}_k= \argmax_{\alpha}\mi(\alpha\,,P_k)$ \hspace{18mm} \alphak\\
    \hspace{4.5mm} solved with Brent's method~\cite{4517235}\\
    }
\KwOut{$\x_k\,,\y_k$}
\caption{Iterative global algorithm}
\label{algo:Mitsos}
\end{algorithm}
\begin{algorithm}[t]
\KwIn{$N_r\,,(\overline{\y}_l)_l\,,\eta\,, k = 0$}
    \While{k $<N_r$}
    {
    $k\leftarrow k+1$ \\
    $\x$ uniformly chosen,
    $\y_{\x} = \argmax_{\overline{\y}_l}\mi\left(\overline{\y}_l\,,\x\right)$\\
        \While{$(\x\,,\y)$ not feasible}
        {
        $\y$ randomly chosen from $\mathcal{N}(\y_{\x}\,,\eta)$
        }
    \textbf{Acceptation:} $k= 0$, with probability $\gamma_{\bm{z},n}\left(\x,\y\right)$
    } % $\bm{z} \leftarrow (\x,\y)$
\KwOut{Accepted $(\x\,,\y)$ giving maximal $\up$}
\caption{Simulated annealing solving \relaxed}
\label{algo:annealing}
\end{algorithm}

\begin{prop}
Algorithm~\ref{algo:Mitsos} stops after a finite number of steps $K$ and delivers an output $(P_K, \alpha_K)$ solution of~\eqref{eq:bilevel}.
\label{prop:mistros}
\end{prop}

%%%%%%%%%%%%%%%%%%%%%%%%%%%%%%%%%%%%%%%%%%%%%%%%%%%%%%%%%%%%%%%%%%%%%%%%%%%%%%%
\subsection{CSO and ENO optimization problems: a simulated annealing approach}
%%%%%%%%%%%%%%%%%%%%%%%%%%%%%%%%%%%%%%%%%%%%%%%%%%%%%%%%%%%%%%%%%%%%%%%%%%%%%%%

Solving the optimization problems~\relaxed~and~\alphak~of Algorithm~\ref{algo:Mitsos} requires a global optimization method for nonconvex and nondifferentiable objectives with continuous constraints.
A natural candidate~\cite{dekkers1991global} is the simulated annealing method introduced in~\cite{romeijn1994simulated}.
The principle is to explore a sufficient number of random feasible couples $\left(\x\,,\y\right)$.
The stopping criterion chosen is based on the concept of acceptance, where a potential couple $\left(\x\,,\y\right)$ is accepted with probability:
\setcounter{equation}{22}
\begin{equation}
        \gamma_{\bm{z},n}\left(\x,\y\right) = \min\hspace{-1mm}\left(\hspace{-1mm}1,\exp\hspace{-1mm}\left(\frac{\up(\x,\bm{L}^*(\y))-\up(\bm{z})}{|\up(\bm{z})|\times K(n)}\right)\hspace{-1mm}\right)
\end{equation}
with $\bm{z}$ the last accepted couple and $K$ a function of the number of iterations $n$, here chosen as $K(n)=0.99^n$.
Note that a couple giving a lower $\up$ than the last accepted couple may be accepted, although it becomes less likely after many iterations (decreasing $K$).
Following~\cite{wah1999simulated}, the algorithm stops when no couple $\left(\x,\y\right)$ has been accepted $N_r$ iterations in a row.

Note that solving scalar optimization~\alphak~is much faster using scalar algorithms like Brent's method~\cite{4517235} rather than simulated annealing.
For problem~\relaxed, the difficulty with simulated annealing is to randomly find feasible couples $\left(\x,\y\right)$, i.e. which verify the constraints in~\relaxed: $\forall l<k,$ $\mi\left(\y, \x, \right) \geq \mi \left(\overline{\y}_l,\x\right) - \frac{\epsmi}{3}$.
% We found that if couples $\left(\x,\y\right)$ are uniformly chosen in $\X\times\Y$, the $\x$ of the feasible couples are not uniformly distributed but favor specific values, whereas for all $\x\in\X$ there exist $\y\in\Y$ so that $\left(\x,\y\right)$ are feasible.
However we observed that the optimum $\overline{\y}$ of $\mi(\cdot,P)$ depends faintly on $P$ because the variations of the electricity supplying costs $C_{i,t}(L_i^*(\alpha),P)$ due to $P$ are small.
Then for every $\x\in\X$, the $\y\in\Y$ such that $\left(\x,\y\right)$ is feasible are in the neighborhood of $\overline{\y}_{\x}$, with $\overline{\y}_{\x} = \argmax_{\overline{\y}_l}\mi\left(\overline{\y}_l,\x\right)$.
Consequently, we suggest that $\x\in\X$ should be uniformly chosen first and then $\y$, drawn from a normal distribution $\mathcal{N}(\overline{\y}_{\x},\eta)$ with mean $\overline{\y}_{\x}$ and standard deviation some parameter $\eta$ to choose.
The resulting simulated annealing method is described in Algorithm~\ref{algo:annealing}.

Our global multilevel problem is solved with Algorithm~\ref{algo:Mitsos}, which uses at each iteration global optimization Algorithm~\ref{algo:annealing}.
This numerical resolution is applied in next section to illustrate how our model can help ENO and CSO make decisions.
œ
%%%%%%%%%%%%%%%%%%%%%%%%%%%%%%%%%%%%%%%%%%%%%%%%%%%%%%%%%%%%%%%%%%%%%%%%%%%%%%%
%%%%%%%%%%%%%%%%%%%%%%%%%%%%%%%%%%%%%%%%%%%%%%%%%%%%%%%%%%%%%%%%%%%%%%%%%%%%%%%
%%%%%%%%%%%%%%%%%%%%%%%%%%%%%%%%%%%%%%%%%%%%%%%%%%%%%%%%%%%%%%%%%%%%%%%%%%%%%%%
\begin{figure}
    \centering
        \includegraphics[width=0.45\textwidth]{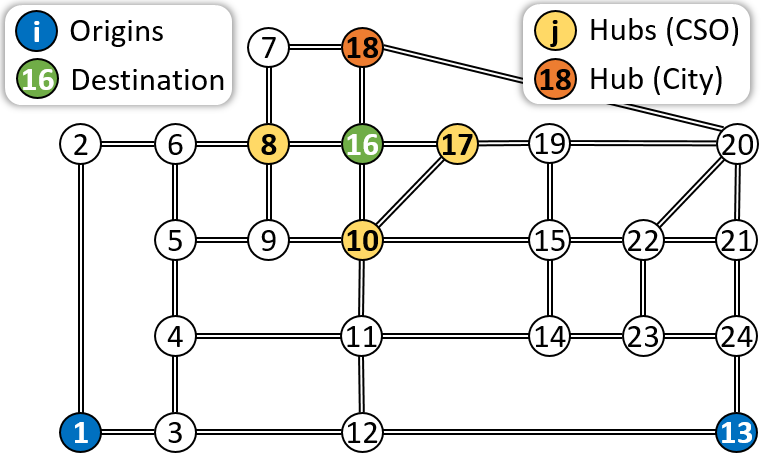}
        \caption{Sioux falls transportation network. \textit{Commuters come from two different origins and have the same destination.
        They choose at which hub to park and the path to get there.}}
\label{fig:transport}
\end{figure}
\section{Case studies}
\label{sec:num}
In this section, Algorithm~\ref{algo:Mitsos} introduced in previous section is applied to our trilevel model to find the optimal strategies for the ENO and the CSO in function of exogenous parameters. %the EV penetration $X_e$.
The parameters of the problem are set as follows, unless otherwise specified: 1500 commuters drive from each origin 1 and 13 (3000 vehicles in total) of Sioux falls transportation network represented in Fig.~\ref{fig:transport}, to destination 16.
More precisely the drivers have to choose at which of the four hubs (at locations 8, 10, 17 and 18, the latter being owned by the city) they want to park and maybe charge.
In Sec.~\ref{sec:Xe}, hubs are supposed equally distant from destination and $t_i=0$ without loss of generality.
The constant charging unit price at city's hub is $\lambda_S^0 = 25$~c\euro/kWh, higher than the one at home, $\lambda_H^0 = 20$~c\euro/kWh.
Half of vehicles are electric (except in Sec.~\ref{sec:Xe}), and the two EV classes $e_0$ and $e_1$ are equally represented, with $s_0 = 5$~kWh and $s_1 = 0$~kWh.
The length of the road between locations 3 and 4 is 2.5~km and the other lengths can be geometrically deduced from it.
For all roads $a$, the speed limit is $v_a=50$~km/h and the road capacity is $C_a=0.2$ (i.e., travel duration triples if 20~\% of the 3,000 vehicles take road $a$).
The values of $\tau=10$~\euro/h, $m_e = 0.2$ kWh/km, $m_g= 0.06$ L/km and $\lambda_g=1.50$~\euro/L are taken from~\cite{sohet2020coupled}.
The four hubs belong to the IEEE 33-bus system illustrated in Fig.~\ref{fig:grid} and whose parameters are given in~\cite{baran1989network}.
In particular, the total nonflexible consumption during working hours near each hub is respectively 1.51, 0.68, 0.45 and 0.45 MWh.
Each hub's total nonflexible consumption is divided into a random profile over $T=8$ time slots.
%Numerical tryouts indicated the following upper bounds for the ENO and the CSO's variables:
The upper bounds for the ENO and the CSO's variables are set high enough to contain the optimal values:
 $\overline{\alpha}=10^{-3}$~\euro/kW$^2$ and $\overline{P} = 4$~MW.
The converting parameters are set as follows: $q=0.1$~\euro/kW$^2$, $\overline{q}=3q$ and $\beta = 10^{-3}$~\euro/kVA$^2$.
\tcb{
Finally, the simulated annealing parameters $N_r = 15$ and $\eta = 2.5\times 10^{-6}$~\euro/kW$^2$ have been adjusted with the help of brute-force search, to ensure a sufficient exploration of $\Pi_{\text{up}}$ domain\footnote{\tcb{For example, above $N_r = 15$, the ratio accepted/explored points is no longer acceptable.}}.
}

\tcb{
Before studying the global trilevel model in the next sections, the aggregated charging profiles~\eqref{eq:water_filling} corresponding to the unique charging need $L^*_i$ at equilibrium at each hub $i$ (see Prop.~\ref{prop:WE_same_costs}) are illustrated in Fig.~\ref{fig:profile}.
This figure shows the local water-filling structure of these profiles (referred to as \textit{Trilevel}) for each CSO's hub.
Figure~\ref{fig:profile} also displays the charging profiles obtained solving~\eqref{eq:grid_aware_wf}, as in the LMP+SC method.
Note that this profile is exclusively concentrated on the third time slot due to a lower nonflexible consumption than during the other time slots.
For comparison, the P\&C profile corresponding to $\bm{L}^*$ is also shown.
It typically leads to significantly larger peak powers compared to the proposed water-filling scheduling and, in turn, higher grid costs.
}

\begin{figure}
    \centering
        \includegraphics[width=0.49\textwidth]{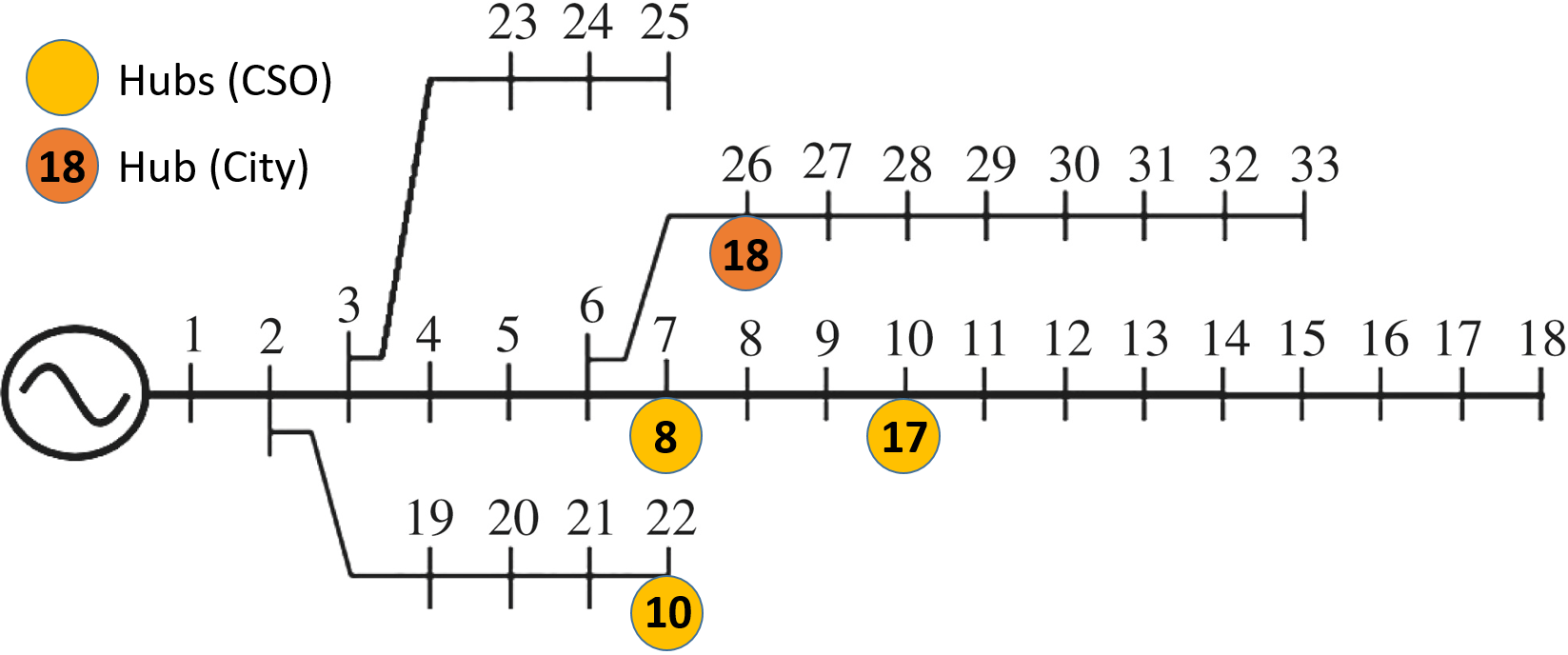}
        \caption{IEEE 33-bus medium-voltage distribution network.}
\label{fig:grid}
\end{figure}
\label{sec:numerical}

\begin{figure}
    \centering
    \includegraphics[width = 0.5\textwidth]{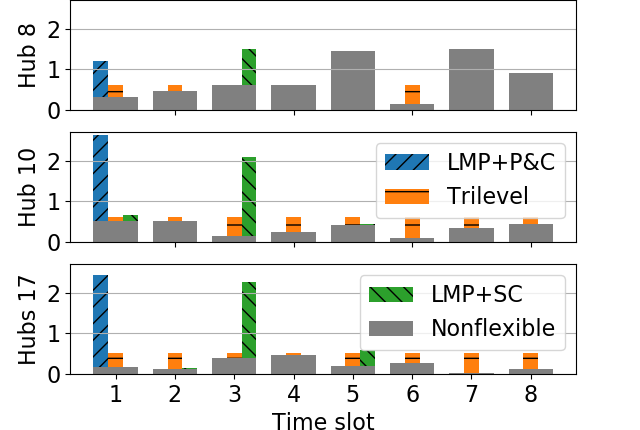}
    \caption{\tcb{Aggregated charging profiles for each CSO's hub, for the water-filling method~\eqref{eq:local} (\textit{Trilevel}), the improved reference method~\eqref{eq:grid_aware_wf} (LMP+SC) and the LMP+P\&C method.}}
    \label{fig:profile}
\end{figure}

\subsection{Sensitivity to Electric Vehicles penetration}
\label{sec:Xe}
\tcg{
Using our trilevel model, operators can find their optimal strategies as the proportion $X_e$ of EVs among vehicles grows.
More precisely, for each $X_e$ value, Algorithm~\ref{algo:Mitsos} gives the corresponding optimal payoffs and strategies for the ENO and the CSO (see Fig.~\ref{fig:obj}).
This figure shows that in general, both payoffs increase with $X_e$, as a higher $X_e$ means more EV charging.
Furthermore, in order to keep affordable charging prices at its hubs, the CSO has to reduce $\alpha$ as $X_e$ increases and amplifies the price incentive part $\mathrm{d}G^*_i/\mathrm{d}L_i$ (see~\eqref{eq:LMP}).
% However, when EV penetration is small ($X_e\leq 20$ \%), smart charging prices at CSO's hubs compensate for the roads congested by GVs so that no EV charges at the city's hub (see Fig.~\ref{fig:Li}).
% These small charging unit prices are due to the small electricity supplying threshold $P^*$ chosen by the ENO, which makes the CSO lower its price strategy to avoid expensive supplying costs.
% In spite of that, the CSO is in deficit when $X_e\leq 15$ \% due the selfish and well-informed decision of the ENO, and must wait for subsidies or a higher EV penetration.
Note that when the number of EVs is high ($X_e\geq85$~\%), the ENO must lower the CSO's contract threshold $P^*$, otherwise the CSO would increase the monetary value $\alpha$ of smart charging to reduce its expensive supplying costs by inciting EVs to rather charge at city's hub.
Thus, the ENO reduces its revenues from CSO's contract so that its payoff stagnates and CSO's payoff considerably increases.
}

\tcg{
%It is also interesting to notice how EVs are divided up among the different hubs.
For each EV penetration $X_e$, there is a unique charging need at each hub corresponding to vehicles' reaction to optimal strategies of the ENO and the CSO.
Note that the uniqueness of the charging need at city's hub is not guaranteed by Prop.~\ref{prop:WE_same_costs}, but is invalidated only in specific cases (e.g., several city's hubs, specific ratios for roads' lengths and energy prices\dots).
As these charging needs greatly increase with $X_e$, they are normalized by the total charging need aggregated over all hubs to emphasize their relative variations: $\tilde{L}_i = L_i/\sum_j L_j$ (see Fig~\ref{fig:Li}).
Note that different temporal profiles of the same nonflexible consumption (at each hub) lead to similar Fig.~\ref{fig:obj}, but different normalized charging needs $\tilde{L}_i$.
Figure~\ref{fig:Li} shows the $\tilde{L}_i$ for two different nonflexible consumption profiles.
This figure reveals that the choice of hub by EVs depends greatly on the nonflexible consumption when the number of EVs is small, but less so as $X_e$ increases.
%Note that the absence of EVs for $X_e\leq 20$ \% at the city's hub is explained in previous paragraph.
As the EV penetration increases, GVs are replaced by EVs, which enables more EVs to use closer hubs to the origins (as 10 and 17), to the detriment of city's hub 18.
Note that fewer EVs choose hub 8 rather than hubs 10 and 17 due to the higher nonflexible consumption there (resp. 1.51 compared to 0.68 and 0.45 MWh).
}

\begin{figure}
    \centering
        \includegraphics[width=0.5\textwidth]{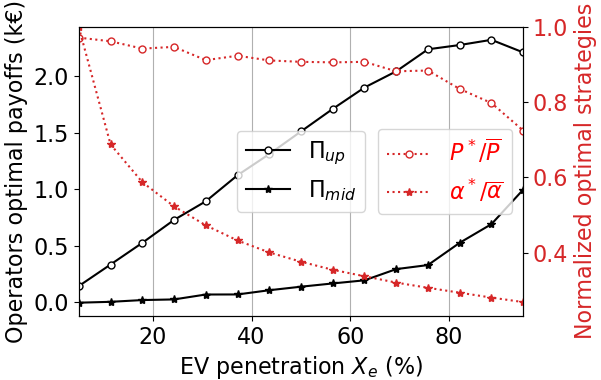}
        \caption{Optimal ENO and CSO's payoffs (resp. $\up$ and $\mi$) and normalized strategies (resp. $P^*/\overline{P}$ and $\alpha^*/\overline{\alpha}$) depending on EV penetration $X_e$.
        \textit{When EV penetration goes over 75~\%, CSO's payoff increases to the detriment of the ENO's because the ENO reduces the CSO's supplying contract.}}
\label{fig:obj}
\end{figure}
\begin{figure}
    \centering
        \includegraphics[width=0.5\textwidth]{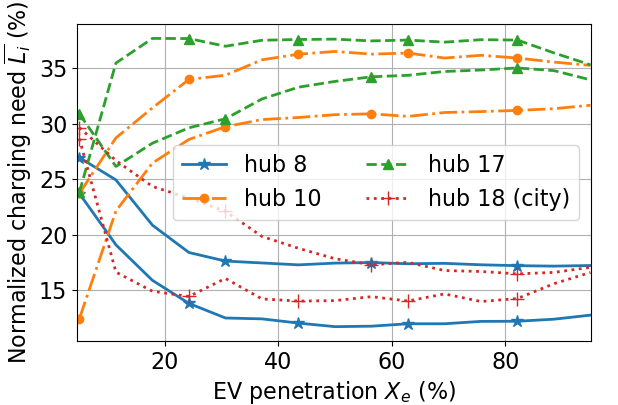}
        \caption{Normalized charging needs $\tilde{L}_i=L_i/\sum_j L_j$ at all hubs depending on EV penetration $X_e$, for two different nonflexible consumption profiles.
        \textit{The profile considered has a significant impact for low EV penetrations.
        For both profiles, $\tilde{L}_8$ and $\tilde{L}_{18}$ decrease with $X_e$ because hub 18 is further away from the origins and hub 8 has a higher nonflexible consumption.}}
\label{fig:Li}
\end{figure}

\begin{figure}
    \centering
        \includegraphics[width=0.5\textwidth]{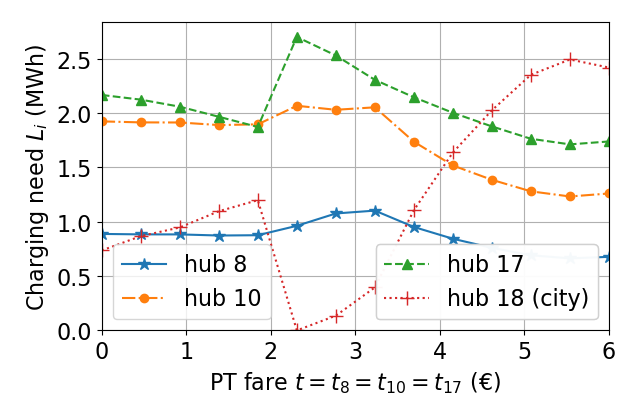}
        \caption{Charging needs $L_i$ at all hubs depending on the unique PT fare $t$.
        \textit{$L_{18}$ globally increases with $t$, except around $t=2$~\euro~where EVs charging at home choose to park at hub 18.}}
\label{fig:LiPT}
\end{figure}

\subsection{Sensitivity to Public Transport fare}
\label{sec:t}
\tcg{
Last section was dedicated to the long-term EV penetration.
This section focuses on the reaction of the ENO and CSO to an incentive coming from the transportation system.
Here, it is supposed that city's hub 18 benefits from a subsidized Public Transport (PT) fare $t_{18} = 1$~\euro.
We consider the PT fare $t$ chosen by a transportation operator and that commuters pay to go from CSO's hubs to the destination: $t=t_8=t_{10}=t_{17}$.
Figure~\ref{fig:LiPT} shows the evolution of charging needs $L_i$ at all hubs $i$ in function of this PT fare.
Note that all EVs of class $e_1$ charge at home: the CSO is better off with high enough charging prices even if it means fewer EVs charging at its hubs.
For PT fares lower than $t=2$~\euro, the number of EVs (of class $e_0$) choosing city's hub increases with $t$. 
Between $t=2$~\euro~and $3$~\euro, this number drops because the PT fare became too expensive for EVs charging at home, which instead all choose paths leading to city's hub 18.
%which can be explained using Fig.~\ref{fig:objPT} which shows the optimal strategies and corresponding payoffs of the ENO and the CSO.
%For small PT fares ($t<2$~\euro), the CSO must maintain a high enough $\alpha$, otherwise too many EVs would charge at CSO's hubs (thanks to low charging and PT costs) and its supplying contract would become too expensive.
%For PT fares between $t=2$ and 3.50~\euro, the higher PT fare will prevent too many EVs charging at CSO's hubs if the CSO lowers $\alpha$.
%Therefore, the CSO can increase its payoff by lowering $\alpha$ in order to have a few more EVs at its hubs, without making its supplying contract too expensive.
Then however, more and more EVs of class $e_0$ naturally choose the city's hub.
%For higher PT fares however, all EVs are chased from city's hub by GVs because of the resulting prohibitive traffic congestion.
%Only unrealistic PT fares ($t\gg 6$~\euro) deter all vehicles from parking at CSO's hubs.
}

\subsection{\tcb{Comparison with iterative method based on literature}}

\tcb{
This section compares the trilevel model built in this paper with the most commonly used model of EV charging incentives in coupled electrical-transportation systems~\cite{alizadeh2016optimal,wei2017network} (see Sec.~\ref{sec:alternative}), on the EV penetration sensitivity example of Sec.~\ref{sec:Xe}.
Figure~\ref{fig:cmp} shows for each EV penetration $X_e$ the grid costs $\mathcal{G}$ (filled black markers) and the charging revenues $\mathcal{R}=\sum_{i\in\mathcal{H}_{\text{cso}}}R_i$ (empty red markers) for the trilevel method (star marker), the improved iterative method (LMP+SC) for two values of $\tilde{\alpha}$, and the LMP+P\&C method for $\tilde{\alpha}=0.01$.
}

\tcb{
Figure~\ref{fig:cmp} shows that for the same $\tilde{\alpha}=0.01$ value, the LMP+P\&C method (diamond marker) gives higher grid costs than the LMP+SC one (square), as expected, but also lower charging revenues:
as grid costs are higher, the charging unit prices too so that EVs prefer to charge at city's hub (up to $X_e=60$~\%, where they accept these high prices because of the congested paths to access city's hub).
The impact of the conversion parameter $\tilde{\alpha}$ is also illustrated in Fig.~\ref{fig:cmp}.
For example, when $\tilde{\alpha}$ is too high (e.g., $\tilde{\alpha}=0.03$), the LMP+SC method (triangle marker) gets similar results as the LMP+P\&C one (diamond).
Note that charging revenues are always higher in the trilevel model of this paper than in the other methods.
This seems intuitive given that this metric is explicitly taken into account in the framework of this paper while the alternative methods focus on grid cost minimization.
%When it is too small (e.g., $\tilde{\alpha}=0.003$, circle marker), the lower price incentive does not reduce grid costs as much as with $\tilde{\alpha}=0.01$ (square), and gives lower revenues.
}

\tcb{
Figure~\ref{fig:cmp} illustrates that the trilevel model of this paper (star marker) obtains fairly low grid costs compared to the LMP+P\&C method or the LMP+SC one with $\tilde{\alpha}$ not carefully designed.
This indicates that the supplying contract, the proxy used in the scheduling problem~\eqref{eq:local} and the corresponding LMP~\eqref{eq:LMP} are good heuristics to reduce grid costs, as expressed in our previous paper~\cite{sohet2020isgt}.
Note that with a particular value $\tilde{\alpha}=0.01$, the LMP+SC method (square marker) obtains the minimal grid costs.
This is made possible because the goal of the operator choosing the charging profiles and prices in this method is to precisely minimize grid costs.
However in practice, the hubs' operator wants to maximize its payoff and may have no information on the electrical grid, as in the trilevel model of the present paper, which guarantees the highest charging revenues among all methods.
Moreover, the results of the LMP+SC method are highly sensitive to the choice of parameter $\tilde{\alpha}$, as shown in Fig.~\ref{fig:cmp}.
Finally, note that parallel computations are not practical for the iterative methods.
Due to the complexity of solving scheduling problem~\eqref{eq:grid_aware_wf}, the LMP+SC method is actually slower (two times in average) to solve than the trilevel model, which has more optimization layers.
% , with lower charging revenues than the trilevel model nonetheless.
% These minimal grid costs are unachievable in practice because the hubs' operator wants to maximize its own payoff and may have no information on the electrical grid, as in our trilevel model.
% Finally, note that parallel computations are not practical for the Ite method, so that the trilevel model which has more optimization layers is actually solved (two times in average) faster.
}

%%%%%%%%%%%%%%%%%%%%%%%%%%%%%%%%%%%%%%%%%%%%%%%%%%%%%%%%%%%%%%%%%%%%%%%%%%%%%%%
%%%%%%%%%%%%%%%%%%%%%%%%%%%%%%%%%%%%%%%%%%%%%%%%%%%%%%%%%%%%%%%%%%%%%%%%%%%%%%%
%%%%%%%%%%%%%%%%%%%%%%%%%%%%%%%%%%%%%%%%%%%%%%%%%%%%%%%%%%%%%%%%%%%%%%%%%%%%%%%
\begin{figure}
    \centering
    \includegraphics[width=0.5\textwidth]{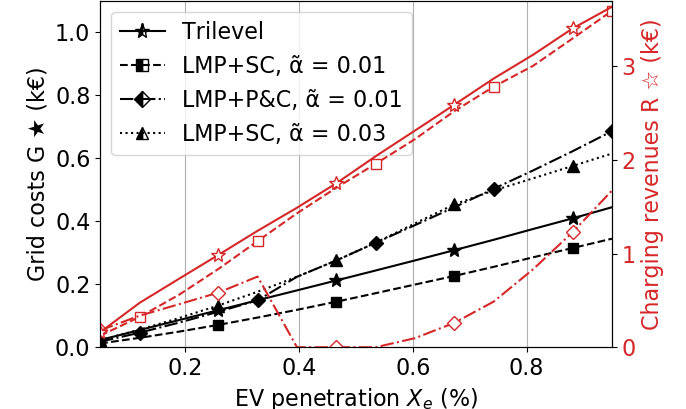}
    \caption{\tcb{ENO grid costs (solid lines) and CSO charging revenues (dashed lines), depending on EV penetration, obtained with our Trilevel method (star marker), the LMP+P\&C one (diamond) and LMP+SC method for different normalizations $\tilde{\alpha}$ of the LMP (square and triangle).
    \textit{
    The literature-based method may lead to minimal grid costs only if smart charging is considered and if $\tilde{\alpha}$ is carefully chosen (square marker).
    Charging revenues are always higher with Trilevel method.}}}
    \label{fig:cmp}
\end{figure}

\section{Conclusion}
In this work, the impact on the electrical system of EV commuting is modeled by a trilevel optimization problem.
The lower, middle and upper levels respectively represent the EVs, interacting in a coupled driving-and-charging congestion game, the CSO which can modify the smart charging prices at its hubs and the ENO which revises the electricity supplying contracts with each hub.
This trilevel problem is seen as a Stackelberg game (between the upper and middle levels) with equilibrium constraints (lower level), which is solved with an optimistic iterative algorithm combined with simulated annealing.
For each ENO and CSO's strategies, we proved that there is a unique charging need at each hub when vehicles are at equilibrium.
The behaviors' coupling between the three levels is illustrated on realistic urban networks, in function of the EV penetration and a transportation incentive.
A comparison with a reference model in the coupled electrical-transportation literature shows the efficiency of the incentives (charging price and supplying contract) in our realistic trilevel model.

In a future work, several CSOs will interact in a game structure, making the trilevel problem an optimization at the upper level, combined with two games both at the middle and the lower levels.
In parallel, a transportation operator (e.g., a public authority responsible for local pollution or dynamic road pricing) will be added to enable a theoretical study of the transportation-electrical coupling.

%%%%%%%%%%%%%%%%%%%%%%%%%%%%%%%%%%%%%%%%%%%%%%%%%%%%%%%%%%%%%%%%%%%%%%%%%%%%%%%
%%%%%%%%%%%%%%%%%%%%%%%%%%%%%%%%%%%%%%%%%%%%%%%%%%%%%%%%%%%%%%%%%%%%%%%%%%%%%%%
%%%%%%%%%%%%%%%%%%%%%%%%%%%%%%%%%%%%%%%%%%%%%%%%%%%%%%%%%%%%%%%%%%%%%%%%%%%%%%%

\appendices

\section{Proof of Prop.~\ref{prop:Beckman}: WE computation}
\begin{proof}
A local minimum of~\eqref{eq:beckmann} verifies the associated Karush–Kuhn–Tucker conditions.
As $\lambda_i$ is proportional to the derivative of $G_i^*$ (see~\eqref{eq:LMP}), these conditions are equivalent to the Definition~\ref{defi:we} of a WE.
See~\cite{sohet2020coupled} for more details.
\end{proof}

\section{Proof of Prop.~\ref{prop:WE_same_costs}: unique charging needs}
\label{appendix:unique}
Proposition~\ref{prop:WE_same_costs} is due to the nondecreasing property of congestion and consumption costs.
The proof of Prop.~\ref{prop:WE_same_costs} requires the following lemma and definition.

\begin{lemme}
For all CSO's strategies $\alpha\in\Y$, the LMP function $\lambda_i^{\alpha}: L_i \mapsto \lambda_i(\alpha, L_i)$ defined in equation (\ref{eq:LMP}) is increasing.
\label{lem:increasing}
\end{lemme}
\begin{proof}
LMP function $\lambda_i^{\alpha}: L_i \mapsto 2\alpha \frac{L_i + L^0_{i,t_0}}{t_0(L_i)}$ is piecewise differentiable for all $\alpha \in\Y$, with derivative $2\alpha/t_0(L_i) > 0$.
We can conclude that $\lambda_i^{\alpha}$ is increasing by showing that it is continuous:
if $L_i^+ = \Delta_t^+$ then $t_0(L_i^+) = t$ and $L_i^+ + L_{i,t}^0 = t\ell^0_{i,t}$ by definition of $\Delta_t$.
Similarly, if $L_i^- = \Delta_t^-$, then $t_0(L_i^-)= t-1$ and $L_i^- + L_{i,t-1}^0 = t\ell^0_{i,t} - L^0_{i,t} + L_{i,t-1}^0 = (t-1)\ell^0_{i,t}$ by definition of $L^0_{i,t}$.
Therefore, $\lambda_i^{\alpha}(L_i^+) = \lambda_i^{\alpha}(L_i^-) = \ell^0_{i,t}$.
\end{proof}

\begin{defi}[\textbf{Variational Inequality}]
Let $Y\subseteq \mathbb{R}^N$ be a
nonempty, closed and convex set.
A vector $\bm{x}\in Y$ is a solution of the Variational Inequality VI($\bm{c}$, Y) if, for any vector $\bm{y}\in Y$:
\begin{equation}
\bm{c}\left(\bm{x}\right)^T \left(\bm{y}-\bm{x}\right) \geq 0\,.
\label{eq:VI}
\end{equation}
\end{defi}

\begin{proof}[Proof of Prop.~\ref{prop:WE_same_costs}]
Let $\bm{x},\bm{y}\in X$ be two WE of game $\mathbb{G}$.
As functions $d_a$ and $\lambda_i$ (Lemma~\ref{lem:increasing}) are increasing, we have:
\begin{equation*}
\left[\bm{c}\left(\bm{x}\right) - \bm{c}\left(\bm{y}\right)\right]^T \left(\bm{x}-\bm{y}\right) =
 \sum_a \left(x_a - y_a\right)\left(d_a(x_a) - d_a(y_a)\right)
\end{equation*}
\begin{equation*}
     + \sum_{i\in\mathcal{H}_{\text{cso}}}\big(L_i(\bm{x})-L_i(\bm{y})\big) \Big(\lambda_i^{\alpha}\big(L_i(\bm{x})\big)-\lambda_i^{\alpha}\big(L_i(\bm{y})\big)\Big) \geq 0
\,,
\end{equation*}
which is equal to 0 if and only if~\eqref{eq:property} holds.

According to~\cite{smith1979existence}, WE $\bm{x}$ and $\bm{y}$ are solutions of $VI(\bm{c}, X)$.
Equation~\eqref{eq:VI} applied to $(\bm{x},\bm{y})$ and $(\bm{y},\bm{x})$ results in:
\begin{equation*}
\big(\bm{c}\left(\bm{x}\right) - \bm{c}\left(\bm{y}\right)\big)^T \left(\bm{x}-\bm{y}\right) ~~\leq ~~0\,.
\vspace{-5mm}
\end{equation*}
\end{proof}
%%%%%%%%%%%%%%%%%%%%%%%%%%%%%%%%%%%%%%%%%%%%%%%%%%%%%%%%%%%%%%%%%%%%%%%%%%%
%%%%%%%%%%%%%%%%%%%%%%%%%%%%%%%%%%%%%%%%%%%%%%%%%%%%%%%%%%%%%%%%%%%%%%%%%%%
%%%%%%%%%%%%%%%%%%%%%%%%%%%%%%%%%%%%%%%%%%%%%%%%%%%%%%%%%%%%%%%%%%%%%%%%%%%
\vspace{-2mm}
\section{Proof of Prop.~\ref{prop:mistros}: convergence of Algorithm~\ref{algo:Mitsos}}
\begin{proof}
\label{appendix:mistros}
%Prop.~\ref{prop:mistros} is equivalent to finding a step $K$ where~\eqref{eq:criter} is verified.
% According to the maximum theorem
% Function $\mi$ defined in~\eqref{eq:mid} is continuous on the compact $\Y\times\X$ as 
According to the maximum theorem (Beckmann function $\mathcal{B}$ continuous), the mapping $\bm{x}^*(\alpha)$ solution of~\eqref{eq:beckmann} is upper hemicontinuous. As for a given $\alpha$, all $\bm{x}^*(\alpha)$ lead to the same $\bm{L}^*(\alpha)$, function $L^*(\alpha)$ and therefore $\mi$ are continuous.
The same theorem states that $\MI(P)$ is continuous because $\mi$ is.
As functions $\mi$ and $\MI$ are continuous respectively on compacts $\Y\times\X$ and $\X$, they are uniformly continuous according to Heine–Cantor theorem, which gives $\delta_{\varepsilon}$ and $\overline{\delta}_{\varepsilon}$ verifying respectively:
% \vspace{-1mm}
% \newpage
% \setcounter{equation}{21}
\begin{equation*}
    \forall (\alpha_0,P_0), (\alpha_1,P_1)\in \Y\times\X~\text{s.t. } \|(\alpha_0,P_0)-(\alpha_1,P_1)\|\leq \delta_{\varepsilon}\,,
\end{equation*}
    \begin{equation}
\mi(\alpha_1\,,P_1) \geq \mi(\alpha_0\,,P_0) - \frac{\epsmi}{3}
\label{eq:uniform}
\end{equation}
\begin{equation*}
  \forall P_0,P_1\in\X ~\text{s.t. } |P_0-P_1|\leq \overline{\delta}_{\varepsilon}\,,~~
     \MI(P_0) \geq \MI(P_1) - \frac{\epsmi}{3}.
\end{equation*}
Let $\delta = \min(\delta_{\varepsilon}, \overline{\delta}_{\varepsilon})$.
As $\mathcal{P}$ is compact, the sequence $\left(P_k\right)$ built at each iteration of Algorithm~\ref{algo:Mitsos} by~\relaxed~admits a subsequence $\left(P_{u(n)}\right)$ which converges to $P_{\text{lim}}$.
Then, by definition:
\vspace{-1mm}
\begin{equation*}
    \exists N_{\delta}\in\mathbb{N}^*~\text{ s.t. }~\forall n\geq N_{\delta}\,,\quad |P_{u(n)}-P_{\text{lim}}|\leq \frac{\delta}{2}\,.
\end{equation*}
Let $k = u(N_{\delta}),K = u(N_{\delta}+1)$.
Then $|P_{k}-P_{K}|\leq \delta$, so that combining~\eqref{eq:uniform} with $(\overline{\alpha}_k,P_{k}),(\overline{\alpha}_k,P_{K})$ gives:
\begin{equation*}
    \mi(\overline{\alpha}_k,P_{K})\geq \mi(\overline{\alpha}_k, P_{k}) -\frac{\epsmi}{3} \geq \MI(P_{K}) - \frac{2}{3}\epsmi,
\end{equation*}
with $\overline{\alpha}_k$ given by~\alphak~at iteration $k$.
Finally, as $(P_{K}, \alpha_{K})$ verifies constraint $l=k$ of~\relaxed, we have:
\begin{equation*}
    \mi(\alpha_{K},P_{K}) \geq \mi(\overline{\alpha}_{k}, P_{K})-\frac{\epsmi}{3}\,,
\end{equation*}
\begin{equation*}
\text{thus} \quad
    \mi(\alpha_{K},P_{K}) \geq\MI(P_{K}) - \epsmi\,,
\end{equation*}
which means that the stopping criteria is reached after iteration $K$, and Algorithm~\ref{algo:Mitsos} ends with $(P_{K}, \alpha_{K})$ solution of~\eqref{eq:bilevel}.
\end{proof}

%%%%%%%%%%%%%%%%%%%%%%%%%%%%%%%%%%%%%%%%%%%%%%%%%%%%%%%%%%%%%%%%%%%%%%%%%%%
%%%%%%%%%%%%%%%%%%%%%%%%%%%%%%%%%%%%%%%%%%%%%%%%%%%%%%%%%%%%%%%%%%%%%%%%%%%
%%%%%%%%%%%%%%%%%%%%%%%%%%%%%%%%%%%%%%%%%%%%%%%%%%%%%%%%%%%%%%%%%%%%%%%%%%%
\vspace{-4mm}
\bibliographystyle{ieeetr} %apalike
\bibliography{myrefs}

%%%%%%%%%%%%%%%%%%%%%%%%%%%%%%%%%%%%%%%%%%%%%%%%%%%%%%%%%%%%%%%%%%%%%%%%%%%
%%%%%%%%%%%%%%%%%%%%%%%%%%%%%%%%%%%%%%%%%%%%%%%%%%%%%%%%%%%%%%%%%%%%%%%%%%%
%%%%%%%%%%%%%%%%%%%%%%%%%%%%%%%%%%%%%%%%%%%%%%%%%%%%%%%%%%%%%%%%%%%%%%%%%%%
\vspace{-10mm}
\begin{IEEEbiography}[{\includegraphics[width=1in,height=1.25in,clip,keepaspectratio]{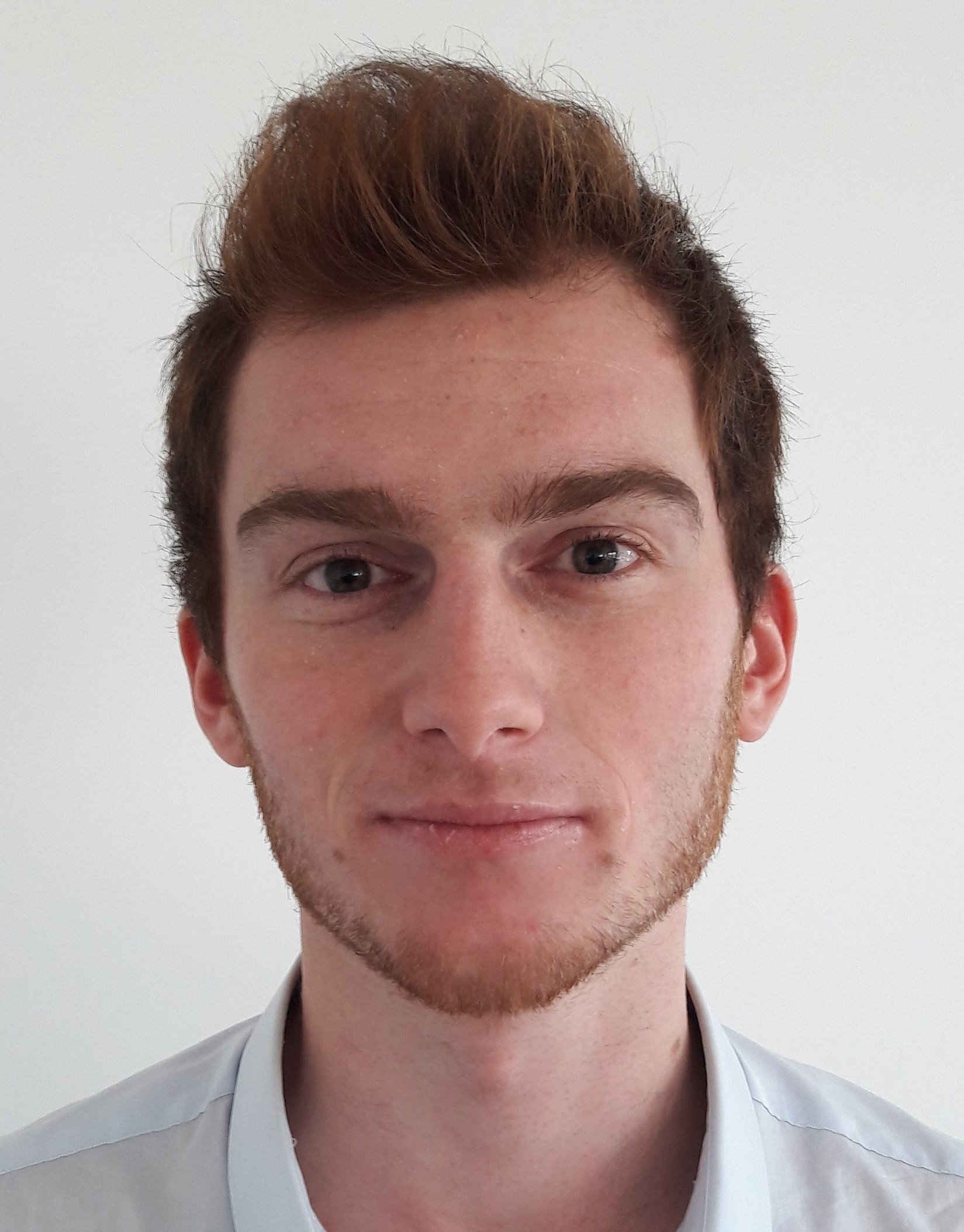}}]{Benoit Sohet}
graduated from \'Ecole Normale Sup\'erieure de Cachan in 2018.
He received the M.Sc. degree in climate physics at Universit\'e de Paris-Saclay in 2017 and the M.Sc. degree in applied mathematics at ENSTA ParisTech in 2018.
He is currently pursuing a Ph.D. in applied mathematics at EDF R\&D and at Avignon Universit\'e.
His research is related to game theory and its applications to the electrical and transportation systems.
\end{IEEEbiography}

\vspace{-10mm}
\begin{IEEEbiography}[{\includegraphics[width=1in,height=1.25in,clip,keepaspectratio]{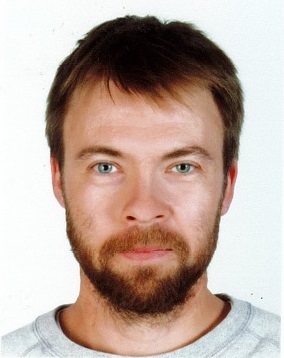}}]{Yezekael Hayel}
(M’08, SM'17) received the M.Sc. degree in computer science and applied mathematics from the University of Rennes~1 in 2002, and the Ph.D. degree in computer science from the University of Rennes~1 and INRIA in 2005. He is an Assistant/Associate Professor with the University of Avignon, France, since 2006. He has held a tenure position (HDR) since 2013. He was a Visiting
Professor with the NYU Polytechnic School of Engineering from 2014 to 2015. He was the Head of the Computer Science/Engineering Institute with the University of Avignon from 2016 to 2019. His research interests include the performance
evaluation and optimization of complex network systems based on game theoretic and queuing models. He was involved at applications in communication/ transportation and social networks, such as wireless flexible networks, bio-inspired and self-organizing networks, and economic models of the Networks. He is associate editor of the GAMES journal.
\end{IEEEbiography}

\vspace{-10mm}
\begin{IEEEbiography}[{\includegraphics[width=1in,height=1.25in,clip,keepaspectratio]{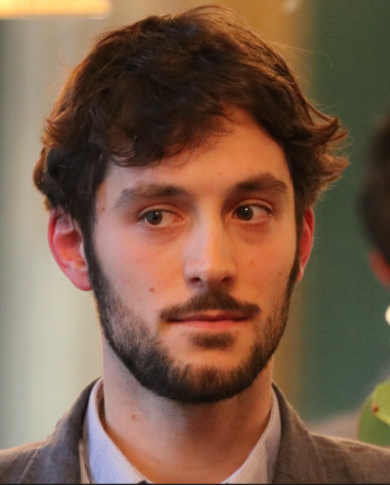}}]{Olivier Beaude}
received the M.Sc. degree in applied mathematics and economics from \'Ecole polytechnique, Palaiseau, France, in 2010, the M.Sc. degree in optimization, game theory, and economic modeling from Universit\'e Paris VI, Paris, France, in 2011, and the Ph.D. degree from the Laboratory of Signals and Systems, CentraleSup\'elec, and the Research and Development Center, Renault, France, in 2015. He is currently a Research Engineer with EDF Research and Development, where his research interests include controlling the impact of EV charging on the grid, and the development of coordination mechanisms, and incentives for local electricity systems.
\end{IEEEbiography}
% \begin{IEEEbiographynophoto}
% \end{IEEEbiographynophoto}

\vspace{-10mm}
\begin{IEEEbiography}[{\includegraphics[width=1in,height=1.25in,clip,keepaspectratio]{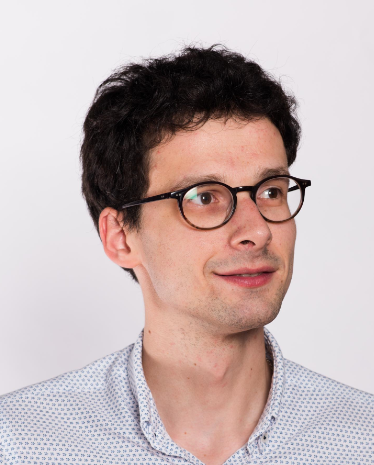}}]{Alban Jeandin}
holds an Engineering degree from T\'el\'ecom Paris and \'Ecole Nationale des Ponts et Chauss\'ees, with specialization in telecommunications, intelligent transportation systems and smart cities. He joined EDF R\&D in 2011. From 2011 to 2016, he was involved in the French smart metering roll out program. He is currently leading EDF R\&D’s research project on vehicle-grid integration and smart charging. His research interests include Smart Metering, Smart Grids, and impact of EV charging on the grid, data exchange standards for Smart Charging and V2G, and Smart Cities.
\end{IEEEbiography}
\vfill

\end{document}